\documentclass[11pt]{article}
\usepackage{latexsym,amsfonts,amssymb,amsmath,amsthm}

\usepackage{color}

\begin{document}

\newtheorem{tm}{Theorem}[section]
\newtheorem{rk}{Remark}[section]
\newtheorem{prop}[tm]{Proposition}
\newtheorem{defin}[tm]{Definition}
\newtheorem{coro}[tm]{Corollary}

\newtheorem{lem}[tm]{Lemma}
\newtheorem{assumption}[tm]{Assumption}

\newtheorem{nota}[tm]{Notation}
\numberwithin{equation}{section}

\newcommand{\stk}[2]{\stackrel{#1}{#2}}
\newcommand{\dwn}[1]{{\scriptstyle #1}\downarrow}
\newcommand{\upa}[1]{{\scriptstyle #1}\uparrow}
\newcommand{\nea}[1]{{\scriptstyle #1}\nearrow}
\newcommand{\sea}[1]{\searrow {\scriptstyle #1}}
\newcommand{\csti}[3]{(#1+1) (#2)^{1/ (#1+1)} (#1)^{- #1
 / (#1+1)} (#3)^{ #1 / (#1 +1)}}
\newcommand{\RR}[1]{\mathbb{#1}}

\newcommand{ \bl}{\color{blue}}
\newcommand {\rd}{\color{red}}
\newcommand{ \bk}{\color{black}}
\newcommand{ \gr}{\color{OliveGreen}}
\newcommand{ \mg}{\color{RedViolet}}

\newcommand{\ep}{\varepsilon}
\newcommand{\rr}{{\mathbb R}}
\newcommand{\alert}[1]{\fbox{#1}}

\newcommand{\eqd}{\sim}
\def\R{{\mathbb R}}
\def\N{{\mathbb N}}
\def\Q{{\mathbb Q}}
\def\C{{\mathbb C}}
\def\l{{\langle}}
\def\r{\rangle}
\def\t{\tau}
\def\k{\kappa}
\def\a{\alpha}
\def\la{\lambda}
\def\De{\Delta}
\def\de{\delta}
\def\ga{\gamma}
\def\Ga{\Gamma}
\def\ep{\varepsilon}
\def\eps{\varepsilon}
\def\si{\sigma}
\def\Re {{\rm Re}\,}
\def\Im {{\rm Im}\,}
\def\E{{\mathbb E}}
\def\P{{\mathbb P}}
\def\Z{{\mathbb Z}}
\def\D{{\mathbb D}}
\newcommand{\ceil}[1]{\lceil{#1}\rceil}

\title{Parabolic-elliptic chemotaxis model with space-time dependent logistic sources on $\mathbb{R}^N$. II. Existence, uniqueness, and stability of strictly
positive entire solutions}

\author{
Rachidi B. Salako and Wenxian Shen\thanks{Partially supported by the NSF grant DMS--1645673} \\
Department of Mathematics and Statistics\\
Auburn University\\
Auburn University, AL 36849\\
U.S.A. }

\date{}
\maketitle
\begin{abstract}
\noindent
The current work is the second of the series of three papers  devoted to the study of asymptotic dynamics in the following parabolic-elliptic chemotaxis system with space and time dependent logistic source,
\begin{equation}\label{main-eq-abstract}
\begin{cases}
\partial_tu=\Delta u -\chi\nabla\cdot(u\nabla v)+u(a(x,t)-ub(x,t)),\quad x\in\R^N,\cr
0=\Delta v-\lambda v+\mu u ,\quad x\in\R^N,
\end{cases}
\end{equation}
where $N\ge 1$ is a positive integer,  $\chi, \lambda$ and $\mu$ are positive constants,
 and the functions $a(x,t)$ and $b(x,t)$ are positive and bounded.  In the first of the series, we studied the phenomena of pointwise and uniform persistence, and asymptotic spreading in  \eqref{main-eq-abstract} for solutions with compactly supported or front like initials. In the second of the series, we investigate the existence, uniqueness and stability of strictly positive entire solutions of \eqref{main-eq-abstract}. In this direction, we prove that, if $0\leq \mu\chi<\inf_{x,t}b(x,t)$, then \eqref{main-eq-abstract} has a  strictly positive entire solution, which is time-periodic (respectively time homogeneous) when the logistic source function is time-periodic (respectively time homogeneous). Next, we show that there is positive constant $\chi_0$, depending on $N$, $\lambda$, $\mu$, $a$ and $b$ such that for every $0\leq \chi<\chi_0$, \eqref{main-eq-abstract} has a unique positive entire solution which is uniform and exponentially  stable with respect to strictly positive perturbations. In particular, we prove that $\chi_0$ can be taken to be $\inf_{x,t}\frac{b(x,t)}{2\mu}$ when the logistic source function is either space homogeneous or the function $(x,t)\mapsto \frac{b(x,t)}{a(x,t)}$ is constant. We also investigate the disturbances to Fisher-KKP dynamics caused by chemotatic effects, and prove that
 \begin{equation*}
\sup_{0<\chi\leq\chi_1}\sup_{t_0\in\R, t\ge 0}\frac{1}{\chi}\|u_{\chi}(\cdot,t+t_0;t_0,u_0)-u_0(\cdot,t+t_0;t_0,u_0)\|_{\infty}<\infty
\end{equation*}
for every $0<\chi_1<\frac{b_{\inf}}{\mu}$ and every uniformly continuous initial function $u_0$, with $\inf_{x}u_{0}(x)>0$, where $ (u_{\chi}(x,t+t_0;t_0,u_0),v_{\chi}(x,t+t_0;t_0,u_0))$ denotes the unique classical solution of \eqref{main-eq-abstract} with $u_{\chi}(x,t_0;t_0,u_0)=u_0(x)$, for every $0\leq \chi<b_{\inf}$.
\end{abstract}



\section{Introduction and Statement of the Main Results}
Chemotaxis, the ability for micro-organisms to respond to chemical signals by moving along the gradient of the chemical substance,   plays important roles in a wide range of biological phenomena (see \cite{ISM04,DAL1991,KJPJAS03}), and accordingly a considerable literature is concerned with its mathematical analysis. We consider the following parabolic-elliptic  chemotaxis system on $\R^N$ with space-time dependent logistic source,
\begin{equation}\label{P}
\begin{cases}
\partial_tu=\Delta u -\chi\nabla\cdot(u\nabla v)+u(a(x,t)-b(x,t)u),\quad x\in\R^N,\cr
0=\Delta v-\lambda v+\mu u ,\quad x\in\R^N,
\end{cases}
\end{equation}
where $u(x,t)$ and $v(x,t)$ denote mobile species density and chemical density functions, respectively,  $\chi$ is a positive constant which measures the sensitivity with respect to chemical signals, $a(x,t)$ and $b(x,t)$ are positive functions  and measure the self growth and self limitation of the mobile species, respectively. The constant $\mu$ is positive and the term $+\mu u$ in the second equation of \eqref{P} indicates that the mobile species produce the chemical substance over time. The positive constant $\lambda$ measures the degradation rate of the chemical substance.  System \eqref{P}  is a space-time logistic source dependant variation of the celebrated parabolic-elliptic Keller-Segel chemotaxis systems (see  \cite{KeSe1, KeSe2}).

Note that \eqref{P} is a particular case of the following chemotaxis model,
\begin{equation}\label{general model}
\begin{cases}
\partial_tu=\Delta u -\chi\nabla\cdot(u\nabla v)+u(a(x,t)-b(x,t)u),\quad x\in\Omega,\cr
\tau v_t=\Delta v-\lambda v+\mu u ,\quad x\in\Omega
\end{cases}
\end{equation}
complemented with certain boundary conditions if $\Omega \subset \R^N$ is a bounded domain, where $\tau\geq 0$ is a nonnegative constant link to the speed of diffusion of the chemical substance. Note that when $\tau=0$ and $\Omega=\R^N$  in \eqref{general model}, we recover \eqref{P}. Hence, \eqref{P}  models the situation where the chemoattractant defuses very quickly and the underlying environment is very large.

It is well known that chemotaxis systems present very interesting dynamics. Indeed, when $\tau >0$, $N=2$, $a(x,t)\equiv b(x,t)\equiv 0$ and \eqref{general model} is considered on a ball centered at origine associated with homogeneous Neumann condition, Herrero and Vel\`azquez \cite{HeVe3} proved the existence of solutions which blow up at finite time. Under these assumptions but $\tau=0$, J\"ager and Lauckhaus \cite{JaLa}  obtained similar results. Similar results were established by Nagai in \cite{Nagai2}.  We also refer the reader to \cite{HeVe1, HeVe2, Dirk and Winkler, KKAS, win_jde, win_JMAA_veryweak, win_arxiv} ( and the references therein) for some other works on the finite-time blow up of solutions of \eqref{general model}. We refer the reader to \cite{BBTW} and the references therein for more insights in the studies of chemotaxis models.

When $a(x,t)>0$ and $b(x,t)>0$, it is known that the blow-up phenomena may be suppressed to some extent. Indeed, if $a(x,t)$ and $b(x,t)$ are  positive constant functions, $\tau=0$ and  $\lambda=\mu=1$ , it is shown in \cite{TeWi2} that if either  $N\leq 2$ or $b>\frac{N-2}{N}\chi$, then for every nonnegative H\"older's continuous initial $u_0(x)$, \eqref{general model} on bounded domain complemented with Neumann boundary condition possesses a unique  bounded global classical solution $(u(x,t;u_0),v(x,t;u_0))$ with $u(x,0;u_0)=u_0(x)$. Furthermore, if $b>2\chi$, then the trivial steady state $(\frac{a}{b},\frac{a}{b})$ is asymptotically stable with respect to nonnegative and non-identically zero perturbations. These results have been extended by the authors of the current paper,  \cite{SaSh1}, to  \eqref{P} on $\R^N$ when $a(x,t)$ and $b(x,t)$ are constant functions. The work \cite{SaSh1} also studied some spreading properties of solutions to \eqref{P} with compactly supported initials.  
Recently, several studies have been concerned with  establishing adequate conditions on the chemotaxis sensitivity $\chi$ and  other parameters in \eqref{general model} to ensure the existence of time global solutions and the stability of equilibria solutions. In this regard, we refer to \cite{ITBWS16, Issa-Shen, OTYM2002, taowin_persistence, WaMuZh,win_CPDE2010}.   The feature of solutions of \eqref{general model} in the presence of logistic type sources still remains a very interesting problem. Indeed, despite such superlinear absorption terms, it  seems to be that blow-up still is possible. The works \cite{win_JMAA_veryweak} should be mentioned in this direction.  It is also worth to note that bounded solutions may exhibit colorful dynamics, characterized by the emergence of arbitrarily large densities in the flavor of \cite{win_JNLS}, and further results of this type have been achieved in \cite{KKAS}. When $a(x,t)$ and $b(x,t)$ are constant positive functions, $\tau=1$ and  $\lambda=\mu=1$, it is shown in \cite{Win} that it is enough for  $\frac{b}{\chi}$ to be sufficiently large to prevent finite time blow up of classical solutions and to guarantee the stability of the constant equilibrium solution $(\frac{a}{b},\frac{a}{b})$.

Thanks to the space and time dependence of the underlying environments in many biological systems,
it is very important to understand the dynamics of the chemotaxis systems with space and time dependent logistic source.
  Note that, when $\chi=0$, the study of \eqref{general model} reduces to the study of the following  equation
\begin{equation}\label{KPP-Fisher equation}
\partial_t u=\Delta u + u(a(x,t)-b(x,t)u), \  x\in \Omega
\end{equation}
complemented with boundary conditions if $\Omega\subset \R^N$ is a bounded domain, which is called the Fisher-KPP equation in literature due to the pioneering works by Fisher  (\cite{Fisher})  and Kolmogorov, Petrowsky, Piscunov (\cite{KPP}). The literature  about the study of \eqref{KPP-Fisher equation} is quite huge.  In a very recent work \cite{ITBWS16}, the authors studied the dynamics of \eqref{general model}  on bounded domains with Neumann boundary conditions and with  space-time dependent logistic source.

The objective of the series of three papers  is to   study the asymptotic dynamics in the chemotaxis system \eqref{P} on the whole space with space and time dependent logistic source.  In the first of the series, we studied the phenomena of pointwise and uniform persistence, and asymptotic spreading in  \eqref{P} for solutions with compactly supported or front like initials. In this second part of the series, we investigate the existence, uniqueness and stability of strictly positive entire solutions of \eqref{P}. In the rest of the introduction, we introduce notations and standing assumptions,  recall some results established in the first part of the series (\cite{SaSh_6_I}), and state the main results of the current part.

\subsection{Notations and standing assumptions}
For every $x\in\R^N$, let $|x|_{\infty}=\max\{|x_i| \ |\ i=1,\cdots,N\}$ and $|x|=\sqrt{|x_1|^2+\cdots+|x_N|^2}$. For every $x\in\R^N$ and $r>0$ we define $B(x,r):=\{y\in\R^N\ |\ |x-y|<r\}$. For every function $w : \R^N\times I\to \R$, where $I\subset \R$, we set $w_{\inf}(t):=\inf_{x\in\R^N}w(x,t)$, $w_{\sup}(t):=\sup_{x\in\R^N}w(x,t)$, $w_{\inf}=\inf_{x\in\R^N,t\in I}w(x,t)$ and $w_{\sup}=\sup_{x\in\R^N,t\in I}w(x,t)$.  In particular,  for every nonnegative  $u_0\in C^{b}_{\rm unif}(\R^N)$, we set $u_{0\inf}=\inf_{x\in\R^N}u_0(x) $ and $u_{0\sup}=\sup_{x\in\R^N}u_0(x)=\|u_0\|_{\infty}$, where
$$
C_{\rm unif}^b(\R^N)=\{u\in C(\R^N)\,|\, u(x)\quad \text{is uniformly continuous in}\,\,\, x\in\R^N\quad \text{and}\,\, \sup_{x\in\R^N}|u(x)|<\infty\}
$$
equipped with the norm $\|u\|_\infty=\sup_{x\in\R^N}|u(x)|$. For any $0\le \nu<1$, let
$$
C_{\rm unif}^{b,\nu}(\R^N)=\{u\in C_{\rm unif}^b(\R^N)\,|\, \sup_{x,y\in\R,x\not = y}\frac{|u(x)-u(y)|}{|x-y|^\nu}<\infty\}
$$
with norm $\|u\|_{C_{\rm unif}^{b,\nu}}=\sup_{x\in\R}|u(x)|+\sup_{x,y\in\R,x\not =y}\frac{|u(x)-u(y)|}{|x-y|^\nu}$. Hence $C_{\rm unif}^{b,0}(\R^N)=C_{\rm unif}^{b}(\R^N)$.

In what follows we shall always suppose that the following hypothesis holds:

\medskip

\noindent {\bf (H)} {\it $a(x,t)$ and  $b(x,t)$ are uniformly H\"older continuous in $(x,t)\in\R^N\times\R$ with exponent $0<\nu<1$ 
and
$$
 0<\inf_{x\in\R^N,t\in\R}\min\{a(x,t), b(x,t)\} \leq \sup_{x\in \R^N,t\in\R}\max\{a(x,t),b(x,t)\}<\infty.
$$
}

    Let $t_0\in\R$ and $T>t_0$. We say that $(u(x,t),v(x,t))$ is a {\it classical solution } of \eqref{P} on $[t_0, T)$ if
 $(u(\cdot,\cdot),v(\cdot,\cdot))\in C(\R^N\times[t_0,T))\cap C^{2,1}(\R^N\times(t_0,T))$
    and satisfies \eqref{P} for $(x,t)\in\R^N\times(t_0,T)$ in the classical sense.
    When  a classical solution  $(u(x,t),v(x,t))$ of \eqref{P} on $[t_0,T)$ satisfies
    $u(x,t)\geq 0$ and $v(x,t)\geq 0$ for every $(x,t)\in\R^N\times[t_0,T)$, we say that it is nonnegative. A {\it global classical solution}  of \eqref{P} on $[t_0,\infty)$ is a classical solution on $[t_0, T)$ for every $T>0$. We say that $(u(x,t),v(x,t))$ is an
   {\it entire solution} of \eqref{P} if $(u(x,t),v(x,t))$ is a global classical solution of \eqref{P} on $[t_0, \infty)$ for every $t_0\in\R.$
     For given  uniformly continuous  function $u_0$ and $t_0,T\in\R$ with $T>t_0$, if $(u(x,t),v(x,t))$ is a classical solution of \eqref{P} with
     $u(x,t_0)=u_0(x)$ for all $x\in\R$, we
     denote it as
    $(u(x,t;t_0,u_0),v(x,t;t_0,u_0))$ and call it the {\it solution of \eqref{P} with initial function
    $u_0(x)$ at time $t_0$}. We shall also use the notation $(u_{\chi}(x,t;t_0,u_0),v_{\chi}(x,t;t_0,u_0))$ to emphasis on the dependence  of solutions to \eqref{P} on the parameter $\chi\geq 0$.

\subsection{Results established in the first part}

 As it is mentioned in the above, in the first part of the series (\cite{SaSh_6_I}), we studied the phenomena of pointwise and uniform persistence, and asymptotic spreading in  \eqref{P} for solutions with compactly supported or front like initials.
Among others,    the following theorems are proved in
\cite{SaSh_6_I}.

\begin{tm}[Global existence]\label{global-existence-tm}\cite[Theorem 1.1]{SaSh_6_I}
Suppose that $\chi\mu \leq b_{\inf},$
then for every $t_0\in\R$ and nonnegative function $u_0\in C^{b}_{\rm unif}(\R^n)\setminus\{0\}$, \eqref{P} has a unique nonnegative global classical solution $(u(x,t;t_0,u_0)$, $v(x,t;t_0,u_0))$
satisfying
$$
\lim_{t\searrow 0}\|u(\cdot,t_0+t;t_0,u_0)-u_0\|_{\infty}=0.
$$ Moreover, it holds that
\begin{equation}\label{u-upper-bound-eq1}
\|u(\cdot,t+t_0;t_0,u_0)\|_{\infty}\leq \|u_0\|_{\infty}e^{ a_{\sup} t}.
\end{equation}
Furthermore, if $${\bf (H1):}\quad b_{\inf}>\chi\mu$$  holds, then the following hold.
\begin{description}
\item[(i)] For every nonnegative initial function $  u_0\in C^{b}_{\rm unif}(\R^N)\setminus\{0\}$ and $t_0\in\R$, there holds
\begin{equation}\label{u-upper-bound-eq2}
0\leq u(x,t+t_0;t_0,u_0)\leq \max\{ \|u_0\|_{\infty}, \frac{a_{\sup}}{b_{\inf}-\chi\mu}\}\,\, \forall\,\, t\ge 0,\,\, \forall\, x\in\R^N,
\end{equation}
and
\begin{equation}\label{u-upper-bound-eq3}
 \limsup_{t\to\infty} \|u(\cdot,t+t_0;t_0,u_0)\|_{\infty}\leq \frac{a_{\sup}}{b_{\inf}-\chi\mu}.
\end{equation}

\item[(ii)]  For every $u_0\in C_{\rm unif}^b(\R^N)$ with
$\inf_{x\in\R^N}u_0(x)>0$ and $t_0\in\R$, there holds
\begin{equation}
\label{asymptotic-lower-bound}
\frac{a_{\inf}}{b_{\sup}}\leq\limsup_{t\to\infty} \sup_{x\in\R^N}u(x,t+t_0;t_0,u_0),\quad \liminf_{t\to\infty}\inf_{x\in\R^N}u(x,t+t_0;t_0,u_0)\le \frac{a_{\sup}}{b_{\inf}}.
\end{equation}

\item[(iii)] For every positive real number $M>0$, there  is a constant $K_{1}=K_1(\nu,M,a,b)$ such that for every $  u_0\in C^{b}_{\rm unif}(\R^N)$ with $0\leq u_0\leq M$, we have
\begin{equation}\label{uniform-holder-bound-for-v}
\|v(\cdot,t+t_0;t_0,u_0)\|_{C^{1,\nu}_{\rm unif}(\R^N)}\leq K_1, \quad \forall \ t_0\in\R,\ \forall\ t\geq 0.
\end{equation}
\end{description}
\end{tm}

\begin{tm}\label{Main-thm1}\cite[Theorem 1.2]{SaSh_6_I}
\begin{itemize}
\item[(i)] (Pointwsie persistence) Suppose that {\bf (H1)} holds. Then pointwise persistence occurs in \eqref{P}, that is,
for any  $u_0\in C_{\rm unif}^b(\mathbb{R}^N)$ with $\inf_{x\in\mathbb{R}^N}u_0(x)>0$
(such $u_0$ is called {\it strictly positive}),
there { exist positive real numbers} $m(u_0)>0$ { and $M(u_0)>0$} such that
 \begin{equation}
 \label{persistence-eq1-0}
 m(u_0)\le u(x,t+t_0;t_0,u_0)\le  M(u_0)\,\, \forall\,\, t_0\in\mathbb{R}\,\, {\rm and}\,\, t>0.
 \end{equation}

 \item[(ii)] (Uniform persistence) 
  Suppose that {\bf (H1)} holds.
 If, furthermore,  $$
{\bf (H2)}: \quad  b_{\inf}>(1+\frac{a_{\sup}}{a_{\inf}})\chi\mu
$$ holds,  then uniform persistence occurs in \eqref{P}, that is,
there are $0<m<M<\infty$ such that for any  $t_0\in\R$ and any positive initial function $u_{0}\in C^{b}_{\rm unif}(\R^N)$ with $\inf_{x\in\R}u_0(x)>0$, there exists $T(u_0)>0$ such that
 $$
 m\le u(x,t+t_0;t_0,u_0)\le M\quad \forall\,\, t\ge { T(u_0),\,\,\forall\,x\in\R^N, \,\, \forall\,\, t_0\in\R.}
 $$

 In particular, for every strictly positive initial $u_0\in C^b_{\rm unif}(\R^N)$
 {\rm (}i.e. $u_0\in C_{\rm unif}^b(\R^N)$ with $u_{0\inf}>0${\rm)} and $\varepsilon>0$, there is $T_{\varepsilon}(u_0)>0$ such that such that the unique classical solution $(u(x,t+t_0;t_0,u_0), v(x,t+t_0;t_0,u_0))$ of \eqref{P} with $u(\cdot,t_0;t_0,u_0)=u_0(\cdot)$ satisfies
\begin{equation}\label{attracting-rect-eq1}
\underline{M}-\varepsilon\leq u(x,t+t_0;t_0,u_0)\leq \overline{M}+\varepsilon, \quad \forall t\geq T_{\varepsilon}(u_0), \ x\in\R^N,\,\,\forall\, t_{0}\in\R
\end{equation} and
\begin{equation}\label{attracting-rect-eq1'}
\frac{\mu\underline{M}}{\lambda}-\varepsilon\leq v(x,t+t_0;t_0,u_0)\leq \frac{\mu\overline{M}}{\lambda}+\varepsilon, \quad \forall t\geq T_{\varepsilon}(u_0), \ x\in\R^N,\,\,\forall\, t_{0}\in\R
\end{equation}
where
\begin{equation}\label{attracting-rect-eq2}
\underline{M}:=\frac{(b_{\inf}-\chi\mu)a_{\inf}-\chi\mu a_{\sup} }{(b_{\sup}-\chi\mu)(b_{\inf}-\chi\mu)-(\chi\mu)^2}>\frac{a_{\inf}-\frac{\chi\mu a_{\sup}}{b_{\inf}-\chi\mu} }{b_{\sup}-\chi\mu}
\end{equation}
and
\begin{equation}\label{attracting-rect-eq3}
\overline{M}:=\frac{(b_{\sup}-\chi\mu)a_{\sup}-\chi\mu a_{\inf} }{(b_{\sup}-\chi\mu)(b_{\inf}-\chi\mu)-(\chi\mu)^2}< \frac{a_{\sup}}{b_{\inf}-\chi\mu}.
\end{equation}
Furthermore, the set
\begin{equation} \label{Invariant set} \mathbb{I}_{inv}:=\{ u\in C^b_{\rm unif}(\R^N)\ : \ \underline{M}\leq u_{0}(x)\leq \overline{M}, \ \forall\, x\in\R^N\}
\end{equation} is a positively invariant set for solutions of \eqref{P}, in the sense that for every $t_0\in\R$ and $u_0\in\mathbb{I}_{inv}$, we have that $u(\cdot,t+t_0;t_0,u_0)\in\mathbb{I}_{inv}$ for every $t\geq0$.
\end{itemize}
\end{tm}

{
\begin{rk}\label{new-rk0}
Using the pointwise persistence established in Theorem \ref{Main-thm1} (i) it can be shown that for any $0<\chi<\frac{b_{\inf}}{\mu}$ and $u_0\in C^b_{\rm unif}(\R^N)$ with $u_{0\inf}>0$, there holds
$$
\limsup_{t\to\infty}\|u(\cdot,t+t_0;t_0,u_0)\|_{\infty}\leq \frac{a_{\sup}-\chi\mu m(u_0)}{b_{\inf}-\chi\mu},
$$
where the limit is uniform in $t_0\in\R$. This  result will be useful when studying asymptotic  dynamics of solutions of \eqref{P} associated to strictly positive initial.
\end{rk}
}

\begin{tm}[Asymptotic spreading]\label{spreading-properties}\cite[Theorem 1.3]{SaSh_6_I}
\begin{itemize}
\item[(1)] Suppose that {\bf (H1)} holds. Then for every $t_0\in \R$ and every nonnegative initial function $u_0\in C^b_{\rm unif}(\R^N)$ with nonempty compact support $supp(u_0)$, we have that
\begin{equation}
\lim_{t\to\infty}\sup_{|x|\geq ct}u(x,t+t_0;t_0,u_0)=0, \quad \forall c> c_{+}^{*},
\end{equation}
where
\begin{equation}
c_{+}^{*}(a,b,\chi,\lambda,\mu):=2\sqrt{a_{\sup}}+ \frac{\chi\mu\sqrt{N}a_{\sup}}{2(b_{\inf}-\chi\mu)\sqrt{\lambda}}.
\end{equation}

\item[(2)]
Suppose that
\begin{equation}
{\bf (H3)}: \ b_{\inf}>\left(1+\frac{\Big(1+\sqrt{1+\frac{Na_{\inf}}{4\lambda}}\Big)a_{\sup}}{2a_{\inf}}\right)\chi\mu.
\end{equation}
Then for every $t_0\in\R$ and nonnegative initial function $u_0\in C^b_{\rm unif}(\R^N)$ with nonempty support $supp(u_0)$, we have that
\begin{equation}\label{lower-bound-spreading-speed}
\liminf_{t\to\infty}\inf_{|x|\leq ct}u(x,t+t_0;t_0,u_0)>0, \quad \forall 0\leq c< c_{-}^{*}(a,b,\chi,\lambda,\mu),
\end{equation}
where
\begin{equation}
c_{-}^*(a,b,\chi,\lambda,\mu):=2\sqrt{a_{\inf}-\frac{\chi\mu a_{\sup}}{b_{\inf}-\chi\mu}}-\chi\frac{\mu\sqrt{N}a_{\sup}}{2\sqrt{\lambda}(b_{\inf}-\chi\mu)}.
\end{equation}
\end{itemize}
\end{tm}

\subsection{Main results of the current part}

 Assume {\bf (H1)}. By Theorems \ref{global-existence-tm} and \ref{Main-thm1}, for any $t_0\in\R$ and strictly positive $u_0\in C_{\rm unif}^b(\R^N)$,
$$
0<\liminf_{t\to\infty}\inf_{x\in\R^N}u(x,t+t_0;t_0,u_0)\le\limsup_{t\to\infty}\sup_{x\in\R^N}u(x,t+t_0;t_0,u_0)<\infty.
$$
Naturally, it is important to know whether there is a {\it strictly positive entire solution}, that is,
an entire solution $(u^+(x,t),v^+(x,t))$ of \eqref{P} with $\inf_{x\in\R^N,t\in\R}u^+(x,t)>0$.
It is also important to know the uniqueness and stability of strictly  entire positive solutions of \eqref{P} (if exist).
We have the following result on the existence of strictly positive entire solutions.

\begin{tm}[Existence of strictly  positive entire solutions]
\label{existence-entire-sol}
Suppose that {\bf (H1)} holds. Then \eqref{P} has a strictly  positive entire solution $(u^+(x,t),v^+(t,x))$. Moreover, the following hold.
\begin{description}
\item[(i)] Any strictly positive entire solution $(u^+(x,t),v^+(x,t))$ of \eqref{P} satisfies
\begin{equation}\label{bound for positive sol}
\frac{a_{\inf}}{b_{\sup}}\leq \sup_{(x,t)\in\R^{N}\times\R}u^+(x,t) \leq \frac{a_{\sup}}{b_{\inf}-\chi\mu}.
\end{equation}
\item[(ii)] If {\bf (H2)} holds, then any strictly  positive entire solution $(u^+(x,t),v^+(x,t))$ of \eqref{P} satisfies
\begin{equation}\label{attracting-rect-eq0}
\underline{M}\leq u^+(x,t)\leq \overline{M},\quad \forall x\in\R^N, \ \,\forall t\in\R.\
\end{equation}
 where $\underline{M}$ and $\overline{M}$ are given by \eqref{attracting-rect-eq2} and \eqref{attracting-rect-eq3}, respectively.

\item[(iii)] If there is $T>0$ such that $a(x,t+T)=a(x,t)$ and $b(x,t+T)=b(x,t)$ for very $x\in\R^N$,  $t\in\R$, then \eqref{P} has a strictly positive entire  solution $(u^+(x,t),v^+(x,t))$ satisfying  $(u^+(x,t+T),v^+(x,t+T))=(u^+(x,t),v^+(x,t))$ for every $x\in\R^N$, $t\in\R$.

\item[(iv)] If $a(x,t)=a(x)$ and $b(x,t)=b(x)$, then \eqref{P} has a strictly positive steady state solution.
\end{description}

\end{tm}

\begin{rk}
\begin{description}
\item[(i)] Theorem \ref{existence-entire-sol} (i) provides  explicit  lower and upper bounds for the supremum of all positive entire solutions. This lower bound is in fact achieved in the case that the functions $a(x,t)$ and $b(x,t)$ are constant.
\item[(ii)] Theorem \ref{existence-entire-sol} (ii) shows that if {\bf (H2)} holds, then the explicit lower bound and upper bound for all positive entire solutions coincide with the lower bound and upper bound of the attraction region given by Theorem \ref{Main-thm1} (ii).
\end{description}

\end{rk}

\medskip

We have the following result on the uniqueness and stability of positive entire solutions of \eqref{P}.

\medskip

\begin{tm}[Uniqueness and stability of strictly positive entire solutions]
\label{Stability of positive space homogeneous steady state solution}
  There is  $\chi_0>0$ such that when $0\le \chi<\chi_0$, there is  $\alpha_{\chi}>0$ such that  \eqref{P} has a unique strictly positive entire solution
$(u_\chi^+(x,t),v_\chi^+(x,t))$ which is uniformly and exponentially stable with respect to strictly positive perturbations in the sense that for any
$u_0\in C_{\rm unif}^b(\R)$ with $u_{0\inf}>0$, there is $M>0$ such that
\begin{equation}\label{new-expo-decay-main-eq1}
\|u(\cdot,t+t_0;t_0,u_0)-u^+_\chi(\cdot,t+t_0)\|_{\infty}\leq Me^{-\alpha_{\chi} t},\ \forall t\geq 0,\,\,\forall\, t_0\in\R
\end{equation}
and
\begin{equation}\label{new-expo-decay-main-eq2}
\|v(\cdot,t+t_0;t_0,u_0)-v^+_\chi(\cdot,t+t_0)\|_{\infty}\leq \frac{\mu}{\lambda}Me^{-\alpha_{\chi} t},\ \forall t\geq 0,\,\,\forall\, t_0\in\R.,\,\,\forall\, t_0\in\R.
\end{equation}

Furthermore, if the logistic function $f(x,t,u)=(a(x,t)-b(x,t)u)u$ is either space homogeneous or is of form $f(x,t,u)=b(x,t)(\kappa -u)u$, $\kappa>0$, then $\chi_0$ can be taken to be $\chi_0=\frac{b_{\inf}}{2\mu}$,  and  $u_{\chi}^+(x,t)=u^+_{0}(t)$, $0<\chi<\chi_0$, is the only stable positive entire solution of the Fisher-KKP equation, \eqref{KPP-Fisher equation}.

\end{tm}

\begin{rk}
\begin{description}
\item[(i)] If we suppose that the logistic function is space homogeneous (resp. the function $\R^N\times\R\ni(x,t)\mapsto \frac{a(x,t)}{b(x,t)}$ is constant), Theorem \ref{Stability of positive space homogeneous steady state solution} establishes the stability  of the unique space homogeneous (resp. space-time homogeneous) strictly positive entire solution of \eqref{P} when the chemotaxis sensitivity satisfies $0<\chi<\frac{b_{\inf}}{2\mu}$.  Furthermore, this results goes beyond the stability of the constant equilibrium given by Theorem \ref{Main-thm1} (ii) when the logistic source is constant,  and show that all positive solutions of \eqref{P} converge exponentially to $(\frac{a}{b},\frac{\mu a}{\lambda b})$ when $0<\chi<\frac{b_{\inf}}{2\mu}$. It should be noted that the hypothesis $0<\chi<\frac{b_{\inf}}{2\mu}$ is weaker than hypothesis {\bf (H2)}.
\item[(ii)] It is worth mentioning that the techniques developed to prove Theorem \ref{Stability of positive space homogeneous steady state solution} can be adopted to study the uniqueness and stability of positive entire solution of \eqref{P}, when \eqref{P} is studied on bounded domains with Neumann boundary conditions. Hence the same result is true in this later case. In particular, the uniqueness and stability of the unique constant equilibrium  solution of \eqref{P}, when studied on bounded domains with  Neumann boundary conditions and $\frac{a(x,t)}{b(x,t)}\equiv $ constant, under the hypothesis $0<\chi<\frac{b_{\inf}}{2\mu}$ improves \cite[Theorem 1.4 (2)]{ITBWS16} in this case.
\end{description}

 \item[(iii)] Let $\chi_0$ be given by Theorem \ref{Stability of positive space homogeneous steady state solution}. One can prove that for every,
 $$ 0<\chi<\min\Big\{\chi_0,\frac{b_{\inf}}{\mu}\Big(1+\frac{\big(1+\sqrt{1+\frac{N a_{\inf}}{4\lambda}}\big)a_{\sup}}{2 a_{\inf}}\Big)^{-1} \Big\},$$ it holds that
  \begin{equation*}
\lim_{t\to\infty}\sup_{|x|\leq ct}|u(x,t+t_0;t_0,u_0)-u^+_{\chi}(x,t)|=0, \quad \forall 0\leq c< c_{-}^{*}(a,b,\chi,\lambda,\mu), \forall\, t_)\in\R
\end{equation*}
whenever $u_0\in C^{b}_{\rm unif}(\R^N)$ is nonnegative with nonempty compact support $supp(u_0)$, where the constant  $c_{-}^{*}(a,b,\chi,\lambda,\mu)$ is given by Theorem \ref{spreading-properties}

\item[(iv)] Incorporating space and/or time  dependence on the logistic source $f(x,t,u)=u(a(x,t)-b(x,t)u)$ adds new challenges in the study of the dynamics of solutions of \eqref{P}. In particular, the existence of strictly positive entire solutions of \eqref{P}
     is very nontrivial to prove when $a$ and $b$ depend on both $t$ and $x$ (note that if $a$ and $b$ are constants, then it follows directly that $(u,v)=(\frac{a}{b}, \frac{\mu}{\lambda}\frac{a}{b})$ is a strictly positive entire solution). Also strictly positive entire solutions
     may depend on the chemotaxis sensibility coefficient $\chi>0$ when the logistic source function depends on time and space. This dependence makes the study of the stability of positive entire solutions much  more difficult than the case that $a$ and $b$ are constants and requires completely new ideas.   Our first step in handling this problem is to first derive some a priori estimates on the space $C^{1}-$norm of positive entire solutions which leads to the definition of the constant $\chi_0$ in Theorem \ref{Stability of positive space homogeneous steady state solution}. Next for $\chi<\chi_0$, by developing a new iterative techniques, a kind of "eventual comparison principle" at each step, we show that the ratio of the solution of \eqref{P} with a strictly positive initial function and a strictly positive entire solution can not be really far away from the constant $ 1$ uniformly in the space variable as the time variable becomes arbitrarily large. Based on this result, in the third step we completes the proof of the result, which also requires new ideas. It is worth mentioning that the techniques developed towards the proof of
     Theorem \ref{Stability of positive space homogeneous steady state solution} can be applied for more general problems.

\end{rk}

\medskip

 We conclude with the following results on the disturbances to Fisher-KPP dynamics caused by weak chemotactic effects. In this direction we have the following result.

\medskip

\begin{tm}\label{perturbation effect}
Assume (H1). Let $(u^+_{\chi}(x,t),v^+_{\chi}(x,t))$ denotes strictly  positive entire solution of \eqref{P} with $0\leq \chi <\frac{b_{\inf}}{\mu}$. Then  it holds that
\begin{equation}\label{small chemotatic pert eq1}
\sup_{t\in\R}\|u^+_\chi(\cdot,t)-u^+_0(\cdot,t)\|_{\infty}\leq \frac{\chi\mu a_{\sup}u^+_{0\sup}K}{(b_{\inf}-\chi\mu)b_{\inf}u^+_{0\inf}},
\end{equation}
where $K:=\left( 2 +\frac{\sqrt{N}}{\sqrt{\lambda}}\sup_{t\in\R}\|\nabla\ln(u^+_0(\cdot,t))\|_{\infty}\right)$. Furthermore, we have that
\begin{equation}\label{small chemotatic pert eq2}
\sup_{0<\chi\leq\chi_1}\sup_{t_0\in\R}\frac{1}{\chi}\Big(\sup_{t\geq 0}\|u_{\chi}(\cdot,t+t_0;t_0,u_0)-u_0(\cdot,t+t_0;t_0,u_0)\|_{\infty} \Big)<\infty,
\end{equation}
for every $0<\chi_1<\frac{b_{\inf}}{\mu}$ and every $u_0\in C^{b}_{\rm unif}(\R^N)$ with $u_{0\inf}>0$.
\end{tm}

\begin{rk} It follows from Theorem \ref{perturbation effect} that
$$
\lim_{\chi\to 0^+}\|u^+_\chi(\cdot,t)-u^+_0(\cdot,t)\|_{\infty}=0
$$
uniformly in $t\in\R$. Furthermore, for every   $u_0\in C^{b}_{\rm unif}(\R^N)$, with $u_{0\inf}>0$, it holds that
$$
\lim_{\chi\to0^+}\|u_{\chi}(\cdot,t+t_0;t_0,u_0)-u_{0}(\cdot,t+t_0;t_0,u_0) \|_{\infty}=0,
$$
uniformly in $t\geq 0$ and $t_0\in\R$.
\end{rk}

The rest of the paper is organized as follows. In section 2, we present some preliminary lemmas.
In Section 3, we study the existence of strictly positive entire solutions and prove Theorem \ref{existence-entire-sol}. We investigate in section 4 the uniqueness and stability of  strictly positive entire solution and prove Theorem \ref{Stability of positive space homogeneous steady state solution}. The proof of Theorem \ref{perturbation effect} is also given in section 4.

\medskip


\section{Preliminary lemmas}

\begin{lem} \label{main-lem1}
 Suppose that {\bf (H1)} holds. Then for every $T>0$, $t_0\in\R,$ and  for every nonnegative initial function $ u_0\in C^{b}_{\rm unif}(\R^N)$, there holds that
\begin{equation}\label{eq1-main-lem1}
\inf_{x\in\R^N}u(x,t+t_0;t_0,u_0)\geq u_{0\inf}e^{t(a_{\inf}-b_{sup}\|u_0\|_{\infty}e^{T a_{\sup}})}, \quad \forall\,\, 0\leq t\leq T.
\end{equation}
In particular for every $T>0$ and  for every nonnegative initial $u_0\in C^b_{\rm unif}(\R^N)$ satisfying $\|u_0\|_{\infty}\leq M_{T}:=\frac{a_{\inf}e^{-a_{\sup}T}}{b_{\sup}}$, we have that
\begin{equation}\label{eq2-main-lem1}
\inf_{x\in\R^N}u(x,t+t_0;t_0,u_0)\geq \inf_{x\in\R^N}u_{0}(x), \quad \forall \,\, 0\leq t\leq T,\ \forall \,\, t_{0}\in\R.
\end{equation}
\end{lem}

\begin{proof}
See \cite[Lemma 3.1]{SaSh_6_I}.
\end{proof}

\begin{lem}\label{continuity with respect to open compact topology}
Assume that (H1) holds. Let $u_0\in C^{b}_{\rm unif}(\R^N)$, $\{u_{0n}\}_{n\geq 1}$ be a sequence of nonnegative functions in $C^{b}_{\rm unif}(\R^N)$, and let $\{t_{0n}\}_{n\geq 1}$ be a sequence of real numbers. Suppose that $0\leq u_{0n}(x)\leq M:=\frac{a_{\sup}}{b_{\inf}-\chi\mu}$ and $\{u_{0n}\}_{n\geq 1}$ converges locally uniformly  to $u_{0}$. Then there exist a subsequence {$\{t_{0n'}\}$ of $\{t_{0n}\}$}, functions $a^*(x,t), b^*(x,t)$  such that $(a(x,t+t_{0n'}),b(x,t+t_{0n'}))\to (a^*(x,t),b^*(x,t))$ locally uniformly as $n'\to\infty$, and  $u(x,t+t_{0n'};t_{0n'},u_{0n'})\to u^*(x,t;0,u_0)$ locally uniformly in $(x,t)$ as $n'\to\infty$, where $(u^*(x,t;0,u_0),v^*(x,t;0,u_0)$ is the classical solution of
\begin{equation*}
\begin{cases}
u_{t}(x,t)=\Delta u(x,t)-\chi\nabla\cdot (u(x,t)\nabla v(x,t))+(a^{*}(x,t)-b^{*}(x,t)u(x,t))u(x,t), \quad x\in\R^N\cr
0=(\Delta-\lambda I)v^*(x,t)+\mu u^*(x,t), \quad x\in\R^N\cr
u^*(x,0)=u_{0}(x), \quad x\in\R^N.
\end{cases}
\end{equation*}
\end{lem}

\begin{proof}
See \cite[Lemma 3.2]{SaSh_6_I}.
\end{proof}

\begin{lem}
\label{main-new-lm1}
Assume that (H1) holds. For every $M>0$, $\varepsilon>0$, and $T>0$,   there exist $L_0=L(M,T,\varepsilon)\gg 1$ and $\delta_0=\delta_0(M,\varepsilon)$ such that for every initial function $u_0\in C^{b}_{\rm unif}(\R^N)$ with $0\leq u_0\leq M$ and   every $L\geq L_0$,
\begin{equation}\label{Local-uniform-boun-for-u}
  u(x,t+t_0;t_0,u_0)\leq \varepsilon,\quad \forall \ 0\leq t\leq T,\ t_0\in\R, \ \forall \,\, |x|_{\infty}<2L
\end{equation}
whenever $ 0\leq u_0(x)\leq \delta_0$ for   $|x|_{\infty}<3L$.
\end{lem}

\begin{proof}
See \cite[Lemma 3.3]{SaSh_6_I}.
\end{proof}

\begin{lem}
\label{main-new-lm2} Assume that (H1) holds.
 For fixed $T>0$, there is $0<\delta_0^*(T)<M^+=\frac{a_{\sup}}{b_{\inf}-\chi\mu}+1$ such that for any $0< \delta\le \delta_0^*(T)$
 and for any $u_0$ with
$\delta\le u_0\le M^+$,
\begin{equation}
\label{aux-eq5-1}
\delta\le u(x,T+t_0;t_0,0,u_0)\le M^+\quad \forall\,\, x\in\R^N, \ \forall\ t_0\in\R.
\end{equation}
\end{lem}

\begin{proof}
See \cite[Lemma 3.5]{SaSh_6_I}.
\end{proof}

While we referred to \cite{SaSh_6_I} for the proof of Lemma \ref{main-new-lm2}, for the sake of clarity in the arguments in the proof of our main result in next section, it is convenient  to point out some fundamental results developed in its proof. Let $a_0=\frac{a_{\inf}}{3}$, $D_{L}:=\{ x\in\R^N \ : \ |x_i| < L \, \forall\ i=1,\cdots,N\}$, and  consider the PDE
\begin{equation}\label{aux-eq1-1}
\begin{cases}
u_{t}-\Delta u-a_0u=0, \quad x\in D_L\cr
u=0 \qquad \quad\qquad x\in\partial D_{L}
\end{cases}
\end{equation}
and its corresponding  eigenvalue problem
\begin{equation}\label{aux-eq1-2}
\begin{cases}
-\Delta u-a_0u=\sigma u, \quad x\in D_L\cr
u=0 \qquad \quad\qquad x\in\partial D_{L}.
\end{cases}
\end{equation}
There is $L_0>1$ such that  the principal eigenvalue of \eqref{aux-eq1-2}, denoted by $\sigma_{_L}$, is negative for every $L\geq L_0$.  Moreover a principal eigenfunction, $\phi_L$, associated to the principal eigenvalue $\sigma_L$ can be chosen such that $0<\phi_L(x)<\phi(0)=1$ for all $x\in D_{L}\setminus\{0\}$ (see \cite{SaSh_6_I}). Moreover for every $0<\varepsilon_0\ll 1$,
{\it  there  is $0<\delta_0\ll 1$ such that for any $u_0\in C_{\rm unif}^b(\R^N)$ with
$0\leq u_0\leq M^+$ and $u_0(x)<\delta_0$ for $|x_i|\le 3L$, $i=1,2,\cdots,N$,
\begin{equation}
\label{aux-eq4-2}
0\le \lambda  v(x,t;t_0,x_0,u_0)\le \frac{a_0}{2\chi},\ \ | \nabla v(x,t;t_0,x_0,u_0)|<\frac{\epsilon_0}{2\chi} \ \ {\rm for}\ \ t_0\le t\le t_0+1,\ x\in D_L,\ x_0\in\R^N
\end{equation}
provided that $L\gg 1$} (see \cite{SaSh_6_I}).

Next we consider the following related periodic-perturbation of \eqref{aux-eq1-1},
\begin{equation}\label{au-eq1-3}
\begin{cases}
u_t-\Delta u -b_{\varepsilon}(x,t)\nabla u - a_0u=0 ,\quad x\in D_L\cr
u=0 \qquad \quad \qquad \qquad\qquad \qquad x\in\partial D_{L}.
\end{cases}
\end{equation}
with $|b_{\varepsilon}(x,t)|\leq \varepsilon$ , $b_{\varepsilon}(x,t+1)=b_{\varepsilon}(x,t)$, and its corresponding periodic eigenvalue problem
\begin{equation}\label{au-eq1-4}
\begin{cases}
u_t-\Delta u -b_{\varepsilon}(x,t)\nabla u - a_0u=\sigma u ,\quad x\in D_L,\ 0<t<1,\cr
u(x,t)=0,  \qquad \quad \qquad \qquad\qquad \qquad x\in\partial D_{L}, 0<t<1,\cr
u(x,0)=u(x,1), \ \ \qquad \qquad\qquad \qquad x\in D_{L}.
\end{cases}
\end{equation}

\noindent We suppose that $b_{\varepsilon}(x,t)$ is $1-$ periodic in $t\in\R$, that is, $b_{\varepsilon}(x,t+1)=b_{\varepsilon}(x,t)$ for all $x\in D_L$, and $t\in\R$,  and  let $U_{L,b_{\varepsilon}}(t,\tau)$, $\tau<t$, denote the solution operator of \eqref{au-eq1-3} on $L^{p}(D_L)$, $N\ll p<\infty$. For, $\tau<t$, the evolution operator $U_{L,\varepsilon}(t,\tau)$ is a compact and strongly positive operator on $W^{2,p}_{0}(D_L):=\{u\in W^{2,p}(D_{L})\ : \ u=0 \ \text{on} \ \partial D_{L}\}$.  Letting $K_{L,b_\varepsilon}:=U_{L,b_{\varepsilon}}(1,0)$, which is compact and strongly positive,  thus its spectrum radius $r_{L,\varepsilon}$, is positive.  By Krein-Rutman Theorem,  $r_{L,\varepsilon}$ is an eigenvalue of $K_{L,\varepsilon}$ with a corresponding positive eigenfunction $u_{L,\varepsilon}$. It is well  known that  $\sigma_{L}^{\varepsilon}:=-\ln(r_{L,\varepsilon})$ is the principal eigenvalue of \eqref{au-eq1-4} with positive 1-periodic eigenfunction  $\phi_{L,\varepsilon}(t)=e^{t\sigma_{L,\varepsilon}}U_{L,b_{\varepsilon}}(t,0)u_{L,\varepsilon}$ (see \cite{Peter Hess}). Note that $U_{L}(t)(\phi_L)=e^{-t\sigma_{L}}\phi_L$, where $U_{L}(t)$ denotes the solution operator of \eqref{aux-eq1-1}. It follows that $K_{L}(\phi_{L})=U_{L}(1)(\phi_L)=e^{-\sigma_{L}}\phi_L$, which implies that $r_{L}\geq e^{-\sigma_L}$. By perturbation theory for parabolic equations, we have that $U_{L,b_{\varepsilon}}(1,0)\to U_{L}(1)$ as $\|b_\varepsilon\|_{C(\bar{D}_{L}\times[0, 1])}\to 0$. Thus, there is $0<\varepsilon_0(L)\ll 1$ such that $r_{L,\varepsilon}\geq e^{-\frac{\sigma_L}{2}}$ whenever  $\|b_\varepsilon\|_{C(\bar{D}_{L}\times[0, 1])}\leq \varepsilon_0(L)$. Hence
$$
\sigma_{L,\varepsilon}=-\ln(r_{L,\varepsilon})\leq \frac{\sigma_L}{2}<0, \quad 0<\varepsilon<\varepsilon_0(L).
$$
Note that $U_{L,b_{\varepsilon}}(t,\tau)\phi_{L,\varepsilon}(\tau)=e^{-(t-\tau)\sigma_{L,\varepsilon}}\phi_{L,\varepsilon}(t)$. Thus for every nonnegative initial $u_0\in C(\overline{D}_{L})$ with $\|u_{0}\|_{\infty}>0$, we have that
\begin{equation}\label{au-eq1-5}
 \sup_{x\in D_L,\tau<t}|(U_{L,b_{\varepsilon}}(t,\tau)u_0)(x)|=\infty, \quad \forall\ \|b_\varepsilon\|_{C(\bar{D}_{L}\times[0, 1])}<\varepsilon_0(L).
\end{equation}

\section{Existence of strictly positive entire solutions}

In this section, we  study the existence of strictly positive entire solutions and prove Theorem \ref{existence-entire-sol}.

\begin{proof}[Proof of Theorem \ref{existence-entire-sol}]

Let $T>0$ be fixed and $\delta_0:=\delta_0^*(T)$ and $M^+=\frac{a_{\sup}}{b_{\inf}-\chi\mu}+1$ be given in Lemma \ref{main-new-lm2}. It follows from Lemma \ref{main-new-lm2}  that
\begin{equation}\label{b-eq0}
\delta_0\leq u(x,T-kT;-kT,u_0)\leq M^+ , \quad\ x\in\R^n, \ k\geq 1, \, \delta_0\leq u_0\leq M^+.
\end{equation} Thus, it follows by induction and uniqueness of solution that
\begin{equation}\label{b-eq01}
\delta_0\leq u(x,nT-kT;-kT,u_0)\leq M^+ , \quad\ x\in\R^n, \ k\geq 1,\, n\geq 1 ,\ \ \delta_0\leq u_0\leq M^+.
\end{equation}
Let $u^k_{n}(x):=u(x,-nT;-kT,\delta_0)$ for all $x\in\R^n$, and $k\geq n\geq  0$.
 Then by a priori estimates for parabolic equations (see \cite{Friedman}), the sequence $\{u^k_{0}\}_{k\geq 1}$ has a locally uniformly convergent subsequence $\{u^{k'}_{0}\}_{k\geq 1}$ to some $u^{*}$ with $u^*\in C^{\nu}_{\rm unif}(\R^n)$ for $0<\nu<1$. Let $u^{+}(x,t)=u(x,t;0,u^*)$ for every $x\in\R^n$ and $t\geq 0$. We claim that $u^{+}(\cdot,\cdot)$ has a backward extension. Indeed, by uniqueness of solution of \eqref{P}, for every $1\leq n\leq k'$, we have that
 \begin{equation}\label{b-eq1}
u^{k'}_{0}(\cdot)=u(\cdot,0;-nT,u(\cdot,-nT;-k'T,\delta_0)) =u(\cdot,0;-nT,u^{k'}_{n}).
 \end{equation}
 Similarly as above, for every $n\geq 1$, there is a function $u^{*}_{n}\in C^{b}_{\rm unif}(\R^n)$ and a subsequence $\{u^{k'_{n}}_{n}\}_{k\geq 1}$ of $\{u^{k'}_{n}\}$ with $u^{k'_{n}}_{n}\rightarrow u^{*}_{n} $ locally uniformly as $k_n^{'}\to\infty$.

 Since $u^{k'_n}_{0}\to u^*$ for each $n\geq 1$ locally uniformly, it follows from \eqref{b-eq1} and Lemma \ref{continuity with respect to open compact topology} that
 $$ u^{*}(\cdot)=u(\cdot,0;-nT,u^{*}_{n}). $$
 Therefore
 \begin{equation}\label{b-eq2}
u^{+}(x,t)=u(x,t;0,u^*)=u(x,t;0,u(\cdot,0;-nT,u^{*}_{n})) =u(x,t;-nT,u^{*}_{n})
 \end{equation}
 for all $x\in\R^N$ and $t\ge 0$.  Note that we have used the uniqueness of solutions of \eqref{P} given by Theorem \ref{global-existence-tm} to derive the last equality in \eqref{b-eq2}.
 Since $u(\cdot,t;-nT,u^{*}_{n})$ is defined for all $t\geq -nT$, then it follows from \eqref{b-eq2} that $u^{+}(x,t)$ has extension to $\R^N\times[-nT, \infty)$ for every $n\in\mathbb{N}$. Therefore, $u^{+}(x,t)$ has a backward extension on $\R^N\times\R$. Note that \eqref{b-eq01} implies that $\delta_0\leq u^*\leq M^+$. Thus, by Theorem \ref{Main-thm1}(i) and Lemma \ref{main-lem1}, we obtain that $0<\inf_{x,t} u^{+}(x,t)\leq \sup_{x,t}u^{+}(x,t)\leq M^+$.
Hence $(u^+(x,t),v^+(x,t))$ is a positive entire solution of \eqref{P}.

\smallskip

 (i) Suppose that  $(u^+(x,t),v^+(x,t))$ is a strictly positive entire solution of \eqref{P}.
  Assume  $u^{+}_{\sup}<\frac{a_{\inf}}{b_{\sup}}$.   Let $T>0$ such that $a_{\inf}>e^{a_{\sup}T}b_{\sup}u^{+}_{\sup}$. It follows from \eqref{eq1-main-lem1} that for every $x\in\R^N$, $t\in\R$, we have
 $$
u^{+}(x,t)=u^{+}(x,t;t-T,u^{+}(\cdot,t-T))\geq e^{T(a_{\inf}-b_{\sup}\|u^{+}(\cdot,t-T)\|_{\infty}e^{a_{\sup}T})}\inf_{y\in\R^N}u^{+}(y,t-T),
 $$
 which implies that
 $$
u^{+}_{\inf}\geq u^{+}_{\inf}e^{T(a_{\inf}-e^{a_{\sup}T}b_{\sup}u^{+}_{\sup})}.
 $$
 This is impossible since $a_{\inf}-e^{a_{\sup}T}b_{\sup}u^{+}_{\sup}>0$. Thus, we must have $u^{+}_{\sup}\geq \frac{a_{\inf}}{b_{\sup}}$.
 Hence the first inequality in \eqref{bound for positive sol} holds.

 Note that
 \begin{equation}\label{b-eq3}
\underline{u}(t-t_0;u^+_{\inf})\leq u^{+}(x,t)\leq \overline{u}(t-t_0;u^{+}_{\sup}), \forall\ t_{0}\in\R, t\geq t_0, \ x\in\R^N,
 \end{equation}
where $\underline{u}(t;u_{\inf}^+)$ solves
$$
\begin{cases}
\frac{d}{dt}\underline{u}=\underline{u}(a_{\inf}-\chi\mu u^+_{\sup}-(b_{\sup}-\chi\mu)\underline{u}),\quad t>0\cr
\underline{u}(0)=u^{+}_{\inf},
\end{cases}
$$ and
$\overline{u}(t;u^+_{\sup})$ solves
$$
\begin{cases}
\frac{d}{dt}\overline{u}=\overline{u}(a_{\sup}-\chi\mu u^+_{\inf}-(b_{\inf}-\chi\mu)\overline{u}),\quad t>0\cr
\underline{u}(0)=u^{+}_{\sup}.
\end{cases}
$$
Note also that
\begin{equation}\label{b-eq4}
\lim_{t_0\to -\infty}\underline{u}(t-t_0;u^+_{\inf})=\frac{(a_{\inf}-\chi\mu{u^+_{\sup}})_{+}}{b_{\sup}-\chi\mu}, \quad  \lim_{t_0\to -\infty}\overline{u}(t-t_0;u^+_{\sup})=\frac{(a_{\sup}-\chi\mu{u^+_{\inf}})_{+}}{b_{\inf}-\chi\mu}.
\end{equation}
Hence, it follows from \eqref{b-eq3} and \eqref{b-eq4} that
\begin{equation}\label{b-e}
\frac{(a_{\inf}-\chi\mu{u^+_{\sup}})_{+}}{b_{\sup}-\chi\mu}\leq u^+_{\inf}\quad \text{and}\quad u^{+}_{\sup}\leq \frac{(a_{\sup}-\chi\mu{u^+_{\inf}})_{+}}{b_{\inf}-\chi\mu}.
\end{equation}
The second inequality of \eqref{b-e} implies that $u^+_{\sup}\le \frac{a_{\sup}}{b_{\inf}-\chi\mu}$. This is the second inequality in \eqref{bound for positive sol}. (i) thus follows.

\smallskip

(ii) Since $0<u^{+}_{\inf}\leq u^+_{\sup}$,  \eqref{b-e}  implies  that
\begin{equation}\label{b-eq5}
(b_{\inf}-\chi\mu)a_{\inf}-\chi\mu a_{\sup}\leq \big((b_{\inf}-\chi\mu)(b_{\sup}-\chi\mu) -(\chi\mu)^2\big)u^+_{\inf}
\end{equation}
and
\begin{equation}\label{b-eq6}
\big((b_{\inf}-\chi\mu)(b_{\sup}-\chi\mu) -(\chi\mu)^2\big)u^+_{\sup}\leq (b_{\sup}-\chi\mu) a_{\sup}-\chi\mu a_{\inf}
\end{equation}
Since {\bf (H2)} holds and $u^{+}_{\inf}>0$, it follows from \eqref{b-eq5} that $(b_{\inf}-\chi\mu)(b_{\sup}-\chi\mu) -(\chi\mu)^2 >0$. Thus \eqref{attracting-rect-eq0} follows from \eqref{b-eq5} and \eqref{b-eq6}.

\smallskip

 (iii) Let $\delta_0^*(T)$ be given by Lemma \ref{main-new-lm2} and $E(T):=\{u\in C^{b}_{\rm unif}(\R^N)\ |\ \delta_{0}^*(T)\leq u_{\inf}\leq u_{\sup}\leq \frac{a_{\inf}}{b_{\sup}-\chi\mu}\}$ endowed with the open compact topology. Lemma \ref{main-new-lm2} implies that the map $\mathbb{P}_{T} :E(T)\ni u_0\mapsto u(\cdot,T;0,u_0)\in E(T)$ is well defined. Note that $E(T)$ is a closed bounded convex  subset of $C^{b}_{\rm uinf}(\R^N)$ endowed with the open compact topology.  Let $\{u_{0n}\}_{n\geq 1}\subset E(T)$ and $u_0\in E(T)$ such that $u_{0n}\to u_0$ uniformly on every compact subset of $\R^N$. For every $n\geq 1$, we have
$$
u_{t}(\cdot,\cdot;0,u_{0n})=\Delta u-\chi\nabla v(\cdot,\cdot;u_{0n})\cdot\nabla u+(a-\chi\lambda v(\cdot,\cdot;0,u_{0n})  -(b-\chi\mu)u)u, \quad t>0
$$
and Theorem \ref{global-existence-tm} (ii) gives
\begin{equation}\label{b-eq7}
\sup_{0\leq t\leq T, n\geq 1}\|v(\cdot,t;0,u_{0n})\|_{C^{1,\nu}_{\rm unif}(\R^N)}<\infty.
\end{equation}
Since $u_{0n}\to u_0$ locally uniformly, it follows from  Lemma \ref{continuity with respect to open compact topology} that there is a subsequence $\{(u(\cdot,\cdot;0,u_{0n'}),v(\cdot,\cdot;0,u_{0n'})) \}_{n\geq 1}$ of $\{(u(\cdot,\cdot;0,u_{0n}),v(\cdot,\cdot;0,u_{0n})) \}_{n\geq 1}$ and a function $(u,v)\in C^{2,1}(\R^{N}\times(0,\infty))$ such that $(u(\cdot,\cdot;0,u_{0n'}),v(\cdot,\cdot;0,u_{0n'}))\to (u,v)$ locally uniformly in $C^{2,1}(\R^N\times(0, \infty))$. Moreover, $ (u,v) $ satisfies $\Delta v-\lambda v+\mu u=0$ and
$$
\begin{cases}
u_{t}=\Delta u-\chi\nabla v\cdot\nabla u+(a-\chi\lambda v  -(b-\chi\mu)u)u, \quad 0<t\leq T\cr
u(0)=u_{0}.
\end{cases}
$$
Thus $(u(x,t),v(x,t))=(u(x,t;0,u_0),v(x,t;0,u_0))$ for every $x\in\R^N$, $t\in[0, T]$. This implies that $u(\cdot,T;0,u_{0n'})\to u(\cdot,T;0,u_0)$ locally uniformly. Hence $\mathbb{P}_{T}$ is continuous.

 Next let $\{u_{0n}\}_{n\geq}\in E(T)$ be given. It follows from \eqref{b-eq7} and a priori estimate for parabolic equations that
 $$
\sup_{n}\|u(\cdot,T;0,u_{0n})\|_{C^{\nu}_{\rm}(\R^N)}<\infty.
 $$ Thus $\{u(\cdot,T;0,u_{0n})\}_{n\geq 1}$ has a convergent subsequence in the open compact topology in $E(T)$. Hence $\mathbb{P}_{T}$ is a compact map. Therefore, Schauder's fixed theorem implies that there is $u^*\in E(T)$ such that $u(\cdot,T;0,u^*)=u^*$.  Clearly $(u(\cdot,\cdot;0,u^*),v(\cdot,\cdot;0,u^*))$ is a $T-$periodic solution of \eqref{P} and can be extended uniquely to a positive entire solution.

\smallskip

  (iv) For every $n\geq 1$, let $T_n=\frac{1}{n}$ and $u_{0n}\in C^{b}_{\rm unif}(\R^N)$, such that $(u(x,t;u_{0n}),v(x,t;u_{0n}))$ is a positive $T_n-$ periodic solution of \eqref{P} with  $\frac{a_{\inf}}{b_{\sup}}\le \sup_{(x,t)}u(x,t;u_{0n})\leq\frac{a_{\sup}}{b_{\inf}-\chi\mu}$.

\smallskip

 \noindent {\bf Claim 1 :}{\it There is $L\gg 1$  such that
 \begin{equation}\label{b-eq8}
 \inf_{n\geq 1, x_0\in\R^N}\sup_{|x|<L}u_{0n}(x+x_0)>0.
\end{equation}
 }

Let $a_0=\frac{a_{\inf}}{3}$ and $L_0\gg 1$ be fixed such that the principal eigenvalue $\lambda_L$ of \eqref{aux-eq1-2} is negative for every $L\geq L_0$. Note that for every nonnegative uniformly continuous function $u_0(x)$ in $D_{L}$, $L\geq L_0$, with $\|u_0\|_{L^\infty(D_L)}>0$, we have that $
\|u(\cdot,t;u_0)\|_{\infty}\to \infty, \ \text{as} \ t\to\infty,$
where $u(x,t;u_0)$ solves the initial- boundary problem \eqref{aux-eq1-1}.
Hence, by \eqref{au-eq1-5},  for ever $L\geq L_0$, there is $\varepsilon_0(L)>0$ such that if $\sup_{x\in D_{L},0\leq t\leq 1}|b_{\varepsilon_0}(x,t)|\leq \varepsilon_0$, $b_{\varepsilon_0}(x,t+1)=b_{\varepsilon_0}(x,t)$  for every $x\in D_{L},\ t\geq0$, then for every nonnegative continuous function $u_0(x)$ on $D_{L}$ with $\|u_0\|_{L^\infty(D_L)}>0$, we have that
\begin{equation}\label{b-eq8'}
\sup_{x\in D_L,t>0}(U_{L,b_{\varepsilon_0(L)}}(t;0)u_0)(x)= \infty.
\end{equation}
where $U_{b_{\varepsilon_0(L)},L}(x,t;0)u_0$
solves the initial boundary value problem \eqref{au-eq1-3}  (with $T=1$).

 Taking $T=1$, $M=\frac{a_{\sup}}{b_{\inf}-\chi\mu}$ and $\varepsilon=\min\{\frac{a_{\inf}}{3(\chi\lambda + b_{\sup}-\chi\mu)},\varepsilon_0(L_0)\} $,  it follows from Lemma \ref{main-new-lm1} and inequalities \eqref{Local-uniform-boun-for-u} and \eqref{aux-eq4-2} that there is $L_1>L_0$  and $\delta_0>0$ such that for every $L\geq L_{1}$
 \begin{equation}\label{b-eq9}
 u(x,t+t_0;u_0)<\varepsilon , \ \ v(x,t+t_0;t_0,u_0)<\varepsilon, \ \ \text{and} \ \ |\nabla v(x,t+t_0;t_0,u_0)|\leq \varepsilon \qquad \forall \,\, 0\leq t\leq 1, \ \ \forall \ x\in D_{L}\
 \end{equation}
 whenever $0\leq u_0(x)\leq \delta_0$ for $|x|\leq  3L,\ \ i=1,\cdots, N.$  Suppose that there is some $n\geq 1$ and $x_0\in\R^N$, such that
 \begin{equation}\label{b-eq10}
 \sup_{|x|_{\infty}< 3L_{1}}u_{0n}(x+x_0)<\delta_0.
 \end{equation}
 Thus, since $(u(x,t;0,u_{0n}),v(x,t;0,u_{0n}))$ is $T_{n}-$periodic with $T_{n}\leq 1$, it follows from \eqref{b-eq9} that
 \begin{align*}
 u_{t}(,\cdot;0,u_{0n})&= \Delta u(\cdot,\cdot;0,u_{0n})-\chi(\nabla v\nabla u)(\cdot,\cdot;0,u_{0n})+ u(a-(b-\chi\mu
)u-\chi\lambda v)\\
&\geq \Delta u(\cdot,\cdot;0,u_{0n})-\chi(\nabla v\nabla u)(\cdot,\cdot;0,u_{0n})+ \frac{a_{\inf}}{3}u(\cdot,\cdot,0,u_{0n}), \quad |x-x_0|<L_1, \ t\geq 0.
 \end{align*}
 Therefore, by comparison principle for parabolic equations, since $L_1\geq L_0$, we have that
 \begin{equation}\label{b-eq11}
 u(x+x_0,t;0,u_{0n})\geq U_{b_{\varepsilon_0},L_0}(x,t;0)u_{0n|_{D_{L_0}}}, \quad \forall\ |x|_{\infty}< L_1, \ \forall t\geq 0
 \end{equation}
 where $u_{0n|_{D_{L_0}}}$ denotes the restriction of $u_{0n}$ on $D_{L_0}$ and $b_{\varepsilon_0}(x,t)=\nabla v(x+x_0,t;0,u_{0n})$ for every $x\in D_{L_0}, t\geq 0$. It follows from \eqref{b-eq8'}  and \eqref{b-eq11} that $\sup_{x,t}u(x,t;0,u_{0n})=\infty$, which is a contradiction.
 Hence claim 1 follows.

 By a priori estimate for parabolic equations, we might suppose that $u_{0n}\to u^* \in C^{b}_{\rm unif}(\R^N)$ in the open compact topology. Let $u^{+}(x,t)=u(x,t;0,u^*)$.

\smallskip

 \noindent {\bf Claim 2:} {\it $u^{+}(x,t)=u^*(x)$ for every $x\in\R^N,$ and $t\geq 0$.}

 Without loss of generality, let us suppose that $u_{0n}\to u^*$ in the open compact topology.  Let  $x\in\R^N$ and $t>0$ be fixed. For every $n\geq 1$, we have that
 \begin{align}\label{b-eq12}
u^{+}(x,t)-u^*(x)&=\underbrace{u(x,t;0,u^*)-u(x,t;0,u_{0n})}_{I_{1,n}(x,t)}+ \underbrace{u(x,t;0,u_{0n})-u(x,[nt]T_n;0,u_{0n})}_{I_{2,n}(x,t)}\cr
&\,\,\,\, +\underbrace{u(x,[nt]T_{n};0,u_{0n})-u^*}_{I_{3,n}(x,t)}.
 \end{align}
 Since $u(x,t;0,u_{0n})$ is $T_n-$periodic, then
 $$I_{3,n}(x,t)=u_{0n}(x)-u^{*}(x)\to 0, \quad \text{as}\ \ n\to\infty $$
  in open compact topology.
 It follows from the variation of constant formula that
 \begin{align*}
I_{2,n}(x,t)&=-\chi\underbrace{\int_{0}^{t-[nt]T_n}e^{(t-[nt]T_n-s)(\Delta-I)}\nabla(u(x,s+[nt]T_n;0,u_{0n})\nabla v(x,s+[nt]T_n;0,u_{0n}))ds}_{I_{2,n}^{1}(x,t)} \nonumber\\
&\,\,\,\, +\underbrace{\int_{0}^{t-[nt]T_n}e^{(t-[nt]T_n-s)(\Delta-I)}\left(((a+1-bu)u)\right)(x,s+[nt]T_n;0,u_{0n})ds}_{I^2_{2,n}(x,t)}
 \end{align*}
 Since $\|u_{0n}\|_{\infty}\leq M$, there is a constant $C$ depending only on $M$ such that
 $$
|I_{2,n}^{1}(x,t)|\leq C\int_{0}^{t-[nt]T_n}(t-[nt]T_n-s)^{-\frac{1}{2}}e^{-(t-[nt]T_n-s)}ds\leq C(t-[nt]T_n)^{\frac{1}{2}}\to0 , \quad \text{as}\ n\to\infty,
 $$
 and
 $$
 |I_{2,n}^{2}(x,t)|\leq C\int_{0}^{t-[nt]T_n}e^{-(t-[nt]T_n-s)}ds=C(1-e^{-(t-[nt]T_n)})\to 0, \quad \text{as}\ n\to\infty.
 $$
 Hence $I_{2,n}(x,t)\to 0$ as $n\to\infty$  in the open compact topology . Since $u_{0n}\to u^{*}$ in the open compact topology, then by Lemma \ref{continuity with respect to open compact topology}, we have that $I_{1,n}(x,t)\to0$ as $n\to\infty$  in open compact topology. Therefore, we conclude from \eqref{b-eq12} that $u^{+}(x,t)=u^{*}(x)$, which complete the proof of Claim 2.

 Next, it follows from Claim 1 that there $L\gg 1$ such that
 \begin{equation}\label{b-eq13}
 \inf_{x_0\in\R^N}\sup_{|x|_{\infty}\leq L}u^*(x+x_0)>0.
 \end{equation}
 Suppose by contradiction that $u^*_{\inf}=0$. Then there is a sequence $\{x_{n}\}_{n\geq 1}$ such that $u^{*}(x_n)\to0$ as $n\to\infty$. Let $u_{n}(x)=u^{*}(x+x_n)$ for every $n\geq 1$. By a prior estimate for parabolic equations, as above, we may suppose that  $u_{n}\to \tilde{u}$ in the open compact topology of $\R^N$ and  $\tilde{u}$ is a steady solution of \eqref{P}. Furthermore, \eqref{b-eq13} implies that $\|\tilde{u}\|_{\infty}>0$. Hence by comparison principle for parabolic equations, we that $\tilde{u}(0)>0$. But $\tilde{u}(0)=\lim_{n\to\infty}u^{*}(x_n)=0$, which impossible. Thus $u_{\inf}^*>0$. Therefore $u^*(x)$ is a positive steady solution of \eqref{P}.
\end{proof}

\section{Uniqueness and stability of strictly positive entire solutions}

In this section we study the uniqueness and stability of strictly  positive entire solutions of \eqref{P}  and prove
 Theorem  \ref{Stability of positive space homogeneous steady state solution} and Theorem \ref{perturbation effect}.  First, we study these questions for general logistic type source function $f(x,t,u)=u(a(x,t)-b(x,t)u)$, and prove that there is a positive constant $\chi_0$ such that for every $ 0\leq \chi<\chi_0$, \eqref{P} has a unique exponentially stable positive entire solution. Next, we examine two frequently encountered cases of logistic source in the literature, mainly space independent logistic source function $f_0(x,t,u)=u(a(t)-b(t)u)$ and logistic source function of the form $f_1(x,t,u)=b(x,t)(\kappa -u)u$, $\kappa>0$, and derive explicit lower bound for $\chi_0$. In this section we shall always assume that {\bf (H1)} holds, so that  pointwise persistence phenomena occurs in \eqref{P} (see Theorem \ref{Main-thm1} (i)). Furthermore, for every  initial function $u_0\in C^b_{\rm unif}(\R^N)$ with $\inf_{x}u_0(x)>0$ and every initial time $t_0\in\R$, it follows from Remark \ref{new-rk0} that there exists  $T_1(u_0)\gg 1$ such that the unique nonnegative global classical solution $(u(x,t+t_0;t_0,u_0), v(x,t+t_0;t_0,u_0))$ of \eqref{P} with $(u(x,t_0;t_0,u_0)=u_0(x)$ satisfies
 \begin{equation}\label{eq-f-0-01}
 0<m(u_0)\leq u(x,t+t_0;t_0,u_0) \leq  \frac{a_{\sup}}{b_{\inf}-\chi\mu}, \forall\ t\geq T_{1}(u_0),\ \forall\ x\in\R^N,\  \forall \,\, t_0\in\R.
 \end{equation}
Henceforth, we shall always suppose  that  $0<u_{0\inf}\leq u_{0\sup}\leq \frac{a}{b_{\inf}-\chi\mu}$.  Note that, by a variation of constant formula, we have that
 \begin{align}\label{eq-f-0-02}
 u(\cdot,t+t_1+t_0;t_0,u_0)&=T(t)u(\cdot,t_1+t_0;t_0,u_0) -\chi\int_0^tT(t-s)\nabla\cdot (u\nabla v)(\cdot,s+t_1+t_0;t_0,u_0)ds\nonumber\\
 & \,\,\, +\int_0^tT(t-s)((a+1-ub)u)(\cdot,s+t_1+t_0;t_0,u_0))ds,
 \end{align}
where $\{T(t)\}_{t\ge 0}$ denotes the analytic semigroup generated by $\Delta-I$ on  $X:=C^b_{\rm unif}(\R^N)$. We  let $X^{\beta}$, $0<\beta\leq 1$, stands for the fractional power spaces associated to $I-\Delta$.
Thus, it holds that (see\cite{Dan Henry}) $X^{\frac{1}{2}+\beta}$ is continuously embedded in $C^{b,1}_{\rm unif}(\R^N)$  with
\begin{equation}\label{new-eq1}
\|\nabla u\|_{C^{b}_{\rm unif}(\R^N)}\leq \frac{\sqrt{N}\Gamma(\beta)}{\sqrt{\pi}\Gamma(\frac{1}{2}+\beta)}\|u\|_{X^{\frac{1}{2}+\beta}}, \ \quad  \forall\ u\in X^{\beta+\frac{1}{2}}, \quad \forall\ 0<\beta<\frac{1}{2},
\end{equation}
\begin{equation}\label{new-eq2}
\|u\|_{C^b_{\rm unif}(\R^N)}\leq \|u\|_{X^{\beta}}, \quad \forall \ u\in X^{\beta}, \quad \forall\ 0<\beta<1,
\end{equation}
and
\begin{equation}\label{new-eq3}
\|T(t)u\|_{X^{\beta}}\leq C_{\beta}t^{-\beta}e^{-t}\|u\|_{\infty}, \quad \ \forall\ t>0, \ \forall\ u\in\ X^{\beta},\ \forall\ 0<\beta<1.
\end{equation}

The next Lemma provides an a priori bound on the sup-norm of gradient of positive entire solutions to \eqref{P}.

\medskip

\begin{lem}\label{bound for gradient of positive entire solution} There is a positive constant $C$ independent of $\chi$, $a$, $b$, $\lambda$ and $\mu$ such that for any positive entire solution $(u^+_{\chi}(x,t),v^+_{\chi}(x,t))$ of \eqref{P}, it holds that
\begin{equation}\label{gradient estimate of positive entire solution}
\|\nabla u^{+}_{\chi}(\cdot,t+t_0)\|_{\infty}\leq C_{\frac{3}{4}}\frac{\sqrt{N}\Gamma(\frac{1}{4})}{\sqrt{\pi}\Gamma(\frac{3}{4})}M_0e^{-t}t^{-\frac{3}{4}}\big(1+CM_2t^{\frac{1}{4}} \big)e^{2t\big(\Gamma(\frac{1}{4})M_2\big)^4}, \quad \forall t_0\in\R, \ \forall t>0,
\end{equation} where $M_0=\frac{a_{\sup}}{b_{\inf}-\chi\mu}$ and $M_{1}=2a_{\sup}+1+\chi\mu M_{0}$ and $M_2:= C_{\frac{3}{4}}\Big(\frac{\chi \mu N\Gamma(\frac{1}{4})}{\sqrt{\pi\lambda}\Gamma(\frac{3}{4})}M_0+M_1\Big) $ and $C_{\frac{3}{4}}$ is given by \eqref{new-eq3}.
\end{lem}
\begin{proof}
Observe from \eqref{eq-f-0-02} that for every $t>0$ and $t_0\in\R$, $u^+_{\chi}(\cdot,t+t_0)$ can be written as
\begin{align}\label{new-eq4}
u^+_{\chi}(\cdot, t+t_0)=&T(t)u^+_{\chi}(\cdot,t_0)-\chi\int_{0}^tT(t-s)(\nabla u^+_{\chi}\cdot\nabla v^+_{\chi})(s+t_0)ds\cr
& +\int_{0}^t T(t-s)\Big((a+1-\chi\lambda v^+_{\chi}-(b-\chi\mu)u^+_{\chi})u^+_{\chi}\Big)(s+t_0)ds.
\end{align}
Note from \cite[Lemma 2.2]{SaSh_6_I} and \eqref{bound for positive sol} ,  that $\|\nabla v^+_{\chi}(\cdot,t+t_0)|_{\infty}\leq \frac{\mu\sqrt{N}}{\sqrt{\lambda}}\|u^+_{\chi}(\cdot,t+t_0)\|_{\infty}\leq \frac{\mu\sqrt{N}}{\sqrt{\lambda}} M_0$. Thus, it follows from \eqref{new-eq3} and \eqref{new-eq1} that
\begin{align*}
& \|\chi\int_{0}^tT(t-s)(\nabla u^+_{\chi}\cdot\nabla v^+_{\chi})(s+t_0)ds\|_{X^{\frac{3}{4}}}\cr
\leq &\frac{\chi \mu \sqrt{N}C_{\frac{3}{4}}M_0}{\sqrt{\lambda}}\int_0^{t}\frac{e^{-(t-s)}}{(t-s)^{\frac{3}{4}}}\|\nabla u^+_{\chi}(\cdot,s+t_0)\|_{\infty}ds\cr
 \leq & \frac{\chi \mu N\Gamma(\frac{1}{4}) C_{\frac{3}{4}}M_0}{\sqrt{\pi\lambda}\Gamma(\frac{3}{4})}\int_0^{t}\frac{e^{-(t-s)}}{(t-s)^{\frac{3}{4}}}\|\nabla u^+_{\chi}(\cdot,s+t_0)\|_{\infty}ds.
\end{align*}
Similarly since $ \lambda\|v^+_{\chi}(\cdot,t+t_\tau)\|_{\infty}\leq \chi\mu \|u^+_{\chi}(\cdot,t+t_\tau)\|_{\infty}\leq \chi\mu M_{0}$, using \eqref{new-eq2} and \eqref{new-eq3},  we obtain
\begin{align*}
&\|\int_{0}^t T(t-s)\Big((a+1-\chi\lambda v^+_{\chi}-(b-\chi\mu)u^+_{\chi})u^+_{\chi}\Big)(\cdot,s+t_0)ds \|_{X^{\frac{3}{4}}}\cr
\leq & C_{\frac{3}{4}}\Big(a_{\sup}+1+\chi\lambda \sup_{\tau}\|v^{+}_{\chi}(\cdot,\tau)\|_{\infty}+(b_{\sup}-\chi\mu)\sup_{\tau}\|u^+_{\chi}(\cdot,\tau)\|_{\infty}\Big)
\int_0^t\frac{e^{-(t-s)}}{(t-s)^{\frac{3}{4}}}\|u^+_{\chi}(\cdot,s+t_0)\|_{\infty}ds\cr
\leq&  C_{\frac{3}{4}}\Big(2a_{\sup}+1+\chi\mu M_0\Big)\int_0^t\frac{e^{-(t-s)}}{(t-s)^{\frac{3}{4}}}\|u^+_{\chi}(\cdot,s+t_0)\|_{\infty}ds.
\end{align*}
Therefore, we have from \eqref{new-eq4} that
\begin{equation*}
\|e^tu^+_{\chi}(\cdot,t+t_0)\|_{X^{\frac{3}{4}}}\leq C_{\frac{3}{4}}M_{0}t^{-\frac{3}{4}}+\underbrace{ C_{\frac{3}{4}}\Big(\frac{\chi \mu N\Gamma(\frac{1}{4})}{\sqrt{\pi\lambda}\Gamma(\frac{3}{4})}M_0+M_1 \Big)}_{:=M_2}\int_0^t\frac{e^{s}\|u^+_{\chi}(\cdot,s+t_0)\|_{\infty}}{(t-s)^{\frac{3}{4}}}ds.
\end{equation*}
Therefore, it follows from \cite[Theorem 3.1.1]{Herbert_Amann} that there is $C>0$ such that
\begin{align*}
\|e^tu^+_{\chi}(\cdot,t+t_0)\|_{X^{\frac{3}{4}}}\leq & C_{\frac{3}{4}}M_0t^{-\frac{3}{4}}\big(1+CM_2t^{\frac{1}{4}} \big)e^{2t\big(\Gamma(\frac{1}{4})M_2\big)^4}.
\end{align*}
Combining this with \eqref{new-eq1}, we obtain \eqref{gradient estimate of positive entire solution}. The Lemma is thus proved.

\end{proof}

\begin{rk}\label{new-rk1}
It follows from Lemma \ref{bound for gradient of positive entire solution} that
$$
\|\nabla u^{+}_{\chi}(\cdot,t)\|_{\infty}=\|\nabla u^+_{\chi}(\cdot,1+(t-1))\|_{\infty}\leq C_{\frac{3}{4}}\frac{\sqrt{N}\Gamma(\frac{1}{4})}{\sqrt{\pi}\Gamma(\frac{3}{4})}M_0e^{-1}\big(1+CM_2 \big)e^{2\big(\Gamma(\frac{1}{4})M_2\big)^4}, \quad  \ \forall t\in\R,
$$
where $C$, $M_0$, $M_1$ and $M_2$ are given by Lemma \ref{bound for gradient of positive entire solution} whenever $(u^+_{\chi}(x,t),v^+_{\chi}(x,t))$ is a positive entire solution of \eqref{P}. Therefore, by setting
\begin{equation}\label{eq-f-0-06}
C_0(\chi):=\sup\{\|\nabla u^{+}_{\chi}(\cdot,t)\|_{\infty}, t\in \R, \ (u^+_{\chi}(x,t),v^+_{\chi}(x,t)) \ \text{is a positive entire solution of \eqref{P}} \},
\end{equation}
we have that $C_0(\chi)<\infty$ for every $0<\chi<\frac{b_{\inf}}{\mu}$. Moreover taking $C_1(\chi)=1+\frac{ C_0(\chi)\sqrt{N}}{u^+_{\chi\inf}\sqrt{\lambda}}$, it follows from \eqref{attracting-rect-eq0} that
\begin{equation*}
\lim_{\chi\to0^+}\frac{\chi\mu C_1(\chi) u^+_{\chi \sup} }{(b_{inf}-\chi\mu)u^+_{\chi \inf}}=0,
\end{equation*}
for any positive entire solution $(u^+_{\chi}(x,t),v^+_{\chi}(x,t))$ of \eqref{P}. Thus, we introduce the following definition
\begin{equation}\label{chi_0 def}
\chi_0:=\sup\{\chi\in( 0, \frac{b_{\inf}}{\mu})\ :\ \forall\  0<\tilde{\chi}<\chi, \ \exists\ (u^+_{\tilde{\chi}},v^+_{\tilde\chi}) \ \text{satisfying}\  \frac{\tilde{\chi}\mu C_{1}(\tilde{\chi})u^+_{\tilde{\chi}\sup}}{(b_{\inf}-\tilde{\chi}\mu)u^+_{\tilde{\chi} \inf}}< 1 \  \}.
\end{equation}
\end{rk}

\begin{lem}
 For given $u_0\in C_{\rm unif}^{b}(\R^N)$  and positive entire solution  $(u^+_{\chi}(x,t),v^+_{\chi}(x,t))$ of \eqref{P} we let
$$
U(x,t+t_0;t_0,u_0):=\frac{u(x,t+t_0;t_0,u_0)}{u_{\chi}^+(x,t+t_0)} \quad \text{and} \quad V(x,t+t_0;t_0,u_0):=\frac{v(x,t+t_0;t_0,u_0)}{v_{\chi}^+(x,t+t_0)}.
$$
Then $U(x,t+t_0;t_0,u_0)$ satisfies
\begin{equation}\label{f-1}
U_t=\Delta U +\nabla U\cdot \nabla(2\ln(u^+_{\chi})-\chi v)+ \chi  \Big(\lambda (v^+_{\chi}-v)+\nabla \ln (u^+_{\chi})\cdot \nabla(v^+_{\chi}-v) \Big)U+(b-\chi\mu)u^+_{\chi}U(1-U).
\end{equation}
In particular, if $u^+_{\chi}(x,t)=u^+_{\chi}(t)$, that is,  $u^+_{\chi}$ is space independent, then
\begin{equation}\label{f-1'}
U_t=\Delta U -\chi\nabla U\cdot \nabla v+ \Big(\chi\lambda (1-V)U+(b-\chi\mu)U(1-U)\Big)u_{\chi}^+(t).
\end{equation}
\end{lem}

\begin{proof} We have that
\begin{equation*}
\begin{split}
U_t=&\frac{1}{(u_{\chi}^+)^2}\Big(u^+_{\chi }\Big(\Delta u -\chi\nabla\cdot(u\nabla v) +(a-b u)u\Big) -u\Big( \Delta u^+_{\chi} -\chi\nabla\cdot(u^+_{\chi}\nabla v^+_{\chi}) +(a-b u^+_{\chi})u^+_{\chi} \Big)  \Big)\cr
=&\frac{1}{u_{\chi}^+}\Big( \Delta u - U\Delta u^+_{\chi} -\chi \Big(\nabla\cdot(u\nabla v)- U\nabla\cdot(u^+_{\chi}\nabla u^+_{\chi}) \Big) \Big) +b u^+_{\chi} U(1-U)\cr
=& \Delta U + 2\nabla U\cdot\nabla \ln(u^+_{\chi})- \frac{\chi}{u_{\chi}^+}\Big(\nabla\cdot(u\nabla v)- U\nabla\cdot(u^+_{\chi}\nabla u^+_{\chi}) \Big) +b u^+_{\chi} U(1-U).
\end{split}
\end{equation*}
On the other, we have
\begin{equation*}
\begin{split}
\nabla\cdot(u\nabla v)-U\nabla\cdot(u^+_{\chi}\nabla v^+_{\chi})
=& \nabla u\cdot\nabla v +Uu^+_{\chi}\Delta v -U\nabla u^+_{\chi}\cdot\nabla v^+_{\chi}-U u^+_{\chi}\Delta v^+ \cr
=& Uu^+_{\chi}\Delta(v-v^+_{\chi})+U\nabla u^+_{\chi}\cdot\nabla(v-v^+_{\chi}) +u^+_{\chi}\nabla U\cdot\nabla v\cr
=& \lambda Uu_{\chi}^+(v-v^+_{\chi}) +\mu(u^+_{\chi})^2U(1-U) +U\nabla u^+_{\chi}\cdot\nabla(v-v^+_{\chi})\cr
&\,  +u^+_{\chi}\nabla U\cdot\nabla v.
\end{split}
\end{equation*}
Hence, \eqref{f-1} holds. \eqref{f-1'} follows directly from \eqref{f-1}.
\end{proof}

We note that,  to show the stability of the positive entire solution $u^+_{\chi}(x,t)$, it is enough to show that $\|U(\cdot,t+t_0;t_0,u_0)-1\|_{\infty}\to 0$ as $t\to\infty$. We first prove the following Theorem, which will be useful for the proof of our main result in this section.

\begin{tm}\label{New-asymp-lem 1}
For every $\varepsilon>0$,  $u_0\in C^b_{\rm unif}(\R^N)$ satisfying $0< u_{0\inf}\leq u_{0\sup}\leq \frac{a_{\sup}}{b_{\inf}-\chi\mu}$, and $n\geq 1$,  there is $T_{\varepsilon,n}>0$  such that
\begin{equation}\label{f-2}
\|U(\cdot,t+t_0;t_0,u_0)-1\|_{\infty}\leq \Big(\frac{\chi\mu C_{1}(\chi)u^+_{\sup}}{(b_{\inf}-\chi\mu)u^+_{\chi \inf}}\Big)^{n}\frac{a_{\sup}}{(b_{\inf}-\chi\mu)u^+_{\chi \sup}} +\varepsilon , \quad \forall\ t\geq T_{n,\varepsilon}, \ t_{0}\in\R,
\end{equation} where $C_1(\chi):=1 +\frac{ C_0(\chi) \sqrt{N}}{u^+_{\chi\inf}\sqrt{\lambda}}$ and $C_0$ is given by \eqref{eq-f-0-06}. Furthermore, if $u_{\chi}^{+}(x,t)=u^+_{\chi}(t)$ is space independent, then $T_{\varepsilon,n}$ can be chosen such that
\begin{equation}\label{f-2'}
\|U(\cdot,t+t_0;t_0,u_0)-1\|_{\infty}\leq \Big(\frac{\chi\mu}{b_{\inf}-\chi\mu}\Big)^{n}\frac{a_{\sup}}{(b_{\inf}-\chi\mu)u^+_{\chi \inf}} +\varepsilon , \quad \forall\ t\geq T_{n,\varepsilon}, \ t_{0}\in\R.
\end{equation}
\end{tm}

\begin{proof} The proof of this theorem is divided in two parts. In the first part, we shall give the proof of the general case. Next, in the second part, we consider the proof of the particular cases.

Since, by \cite[Lemma 2.2]{SaSh_6_I}, $
\|\nabla(v-v^+_{\chi})(\cdot,t+t_0;t_0,u_0) \|_{\infty}\leq \frac{\mu\sqrt{N}}{\sqrt{\lambda}} \|(u-u^+_{\chi})(\cdot,t+t_0;t_0,u_0)\|_{\infty}$ and $\|\lambda(v-v^+_{\chi})(\cdot,t+t_0;t_0,u_0) \|_{\infty} \leq  \mu \|(u-u^+_{\chi})(\cdot,t+t_0;t_0,u_0)\|_{\infty} $ for every $t\geq 0$,  we have  from Remark \ref{new-rk1} that
\begin{align}\label{f-3}
&\|(\lambda (v-v^+_{\chi})+\nabla \ln(u^+_{\chi})\cdot \nabla(v-v^+_{\chi}))(\cdot,t+t_0;t_0,u_0)\|_{\infty}\nonumber\\
\leq &  \underbrace{\Big(1 +\frac{ C_0(\chi) \sqrt{N}}{u^+_{\chi\inf}\sqrt{\lambda}}\Big)}_{=C_1(\chi)}\mu \|(u-u^+_{\chi})(\cdot,t+t_0;t_0,u_0)\|_{\infty} \nonumber\\
\leq &\mu C_1(\chi)u^+_{\chi \sup}(t+t_0) \|(U-1)(\cdot,t+t_0;t_0,u_0)\|_{\infty}, \forall\ t\geq 0,
\end{align}
where $C_0(\chi)$ is given by \eqref{eq-f-0-06}. Observe that from Theorem \ref{global-existence-tm} (i) and Theorem \ref{existence-entire-sol} (i) that
$$
\|u(\cdot,t+t_0;t_0,u_0)-u^+_{\chi}\|_{\infty}\leq \frac{a_{\sup}}{b_{\inf}-\chi\mu}, \quad \forall t\geq 0, \ \forall \ t_0\in\R
$$
Thus it follows from the first inequality in \eqref{f-3} that
$$
\|(\lambda (v-v^+_{\chi})+\nabla \ln(u^+_{\chi})\cdot \nabla(v-v^+_{\chi}))(\cdot,t+t_0;t_0,u_0)\|_{\infty} \leq \frac{\mu C_1(\chi) a_{\sup}}{b_{\inf}-\chi\mu}, \quad \forall t\geq 0, \ \forall \ t_0\in\R.
$$
This combined with  \eqref{f-1} yields that
\begin{equation}\label{f-5}
U_t\leq \Delta U +\nabla U\cdot \nabla(2\ln(u^+_{\chi})-\chi v) + \frac{\chi\mu C_{1}(\chi)a_{\sup}}{b_{\inf}-\chi\mu}U + (b-\chi\mu)u^+_{\chi}U(1-U),
\end{equation}
and
\begin{equation}\label{f-6}
U_t\geq \Delta U +\nabla U\cdot \nabla(2\ln(u^+_{\chi})-\chi v) - \frac{\chi\mu C_{1}(\chi)a_{\sup}}{b_{\inf}-\chi\mu}U + (b-\chi\mu)u^+_{\chi}U(1-U).
\end{equation}

Let $\underline{U}_{1}(t)$  denotes the solutions of the ODE
$$
\begin{cases}
\frac{d \underline{U}}{dt}=-\frac{\chi\mu C_{1}(\chi)a_{\sup}}{b_{\inf}-\chi\mu}\underline{U} + (b_{\inf}-\chi\mu)u^+_{\chi \inf}\underline{U}(1-\underline{U})\cr
\underline{U}(0)=\min\{\frac{u_{0\inf}}{u^+_{\chi \sup}},1\}
\end{cases}
$$ and
$\overline{U}_{1}(t)$  denotes the solutions of the ODE
$$
\begin{cases}
\frac{d \overline{U}}{dt}=\frac{\chi\mu C_{1}(\chi)a_{\sup}}{b_{\inf}-\chi\mu}\overline{U} +(b_{\inf}-\chi\mu)u^+_{\chi \inf}\overline{U}(1-\overline{U})\cr
\overline{U}(0)=\max\{\frac{u_{0 \sup}}{u^+_{\chi \inf}},\frac{(b_{\inf}-\chi\mu)u^+_{\chi \inf}+\frac{\chi\mu C_{1}(\chi)a_{\sup}}{b_{\inf}-\chi\mu}}{(b_{\inf}-\chi\mu)u^+_{\chi \inf}}\}.
\end{cases}
$$
Thus, it follows from comparison principle for ODE's that
\begin{equation}\label{f-7}
 \overline{U}_{1}(t)\geq 1+ \frac{\frac{\chi\mu C_{1}(\chi)a_{\sup}}{b_{\inf}-\chi\mu}}{(b_{\inf}-\chi\mu)u^+_{\chi \inf}}\quad \text{and}\quad 0< \underline{U}_{1}(t)\leq 1 \quad \forall\ t\geq 0.
\end{equation}
Furthermore, it holds that
\begin{equation}\label{f-8}
\lim_{t\to\infty}\underline{U}_{1}(t)=\Big(1- \frac{\frac{\chi\mu C_{1}(\chi)a_{\sup}}{b_{\inf}-\chi\mu}}{(b_{\inf}-\chi\mu)u^+_{\chi \inf}}\Big)_+ \quad \text{and}\quad \lim_{t\to\infty}\overline{U}_{1}(t)=1+\frac{\frac{\chi\mu C_{1}(\chi)a_{\sup}}{b_{\inf}-\chi\mu}}{(b_{\inf}-\chi\mu)u^+_{\chi \inf}}.
\end{equation}
 We claim  that
\begin{equation}\label{f-9}
\underline{U}_{1}(t)\leq U(x,t+t_0;t_0,u_0)\leq \overline{U}_{1}(t), \quad \forall\ x\in\R,\forall\ t\geq 0, \ \forall\ t_0\in\R.
\end{equation}
Indeed, by setting $$\mathcal{L}_{1}^+(U):=\Delta U +\frac{\chi\mu C_{1}(\chi)a_{\sup}}{b_{\inf}-\chi\mu}U+(b(x,t+t_0)-\chi\mu)u^{+}_{\chi}U(1-U)$$ and $$\mathcal{L}_2^-(U):=\Delta U -\frac{\chi\mu C_{1}(\chi)a_{\sup}}{b_{\inf}-\chi\mu}U+(b(x,t+t_0)-\chi\mu)u^{+}_{\chi}U(1-U),$$ it follows from \eqref{f-7} that
\begin{equation}\label{f-10}
\frac{d \overline{U}_1}{dt}-\mathcal{L}_1^+(\overline{U}_1)=((b_{\inf}-\chi)u_{\chi \inf}^+-(b(x,t+t_0)-\chi\mu)u^+_{\chi})\overline{U}_1(1-\overline{U}_1)\geq 0
\end{equation}
and
\begin{equation}\label{f-11}
\frac{d \underline{U}_1}{dt}-\mathcal{L}_1^-(\underline{U}_1)=((b_{\inf}-\chi)u_{\chi \inf}^+-(b(x,t+t_0)-\chi\mu)u^+_{\chi})\underline{U}_1(1-\underline{U}_1)\leq 0.
\end{equation}
Therefore, using \eqref{f-5}, \eqref{f-6}, \eqref{f-10}, \eqref{f-11}, and comparison principle for parabolic equations, we deduce that \eqref{f-9} holds. Thus, it follows from   \eqref{f-8} and \eqref{f-9} that  for every $\varepsilon>0$  there is $T_{1,\varepsilon} \gg 1$ such that
$$
1-\frac{\frac{\chi\mu C_{1}(\chi)a_{\sup}}{b_{\inf}-\chi\mu}}{(b_{\inf}-\chi\mu)u^+_{\chi \inf}}-\varepsilon\leq U(x,t+t_0;t_0,u_0)\leq 1+\frac{\frac{\chi\mu C_{1}(\chi)a_{\sup}}{b_{\inf}-\chi\mu}}{(b_{\inf}-\chi\mu)u^+_{\chi \inf}}+\varepsilon, \quad \forall\ t\geq T_{1,\varepsilon}, \forall\ x\in\R^N, \ t_0\in\R,
$$
which is equivalent to
$$
\|U(\cdot,t+t_0;t_0,u_0)-1\|_{\infty}\leq \frac{\frac{\chi\mu C_{1}(\chi)a_{\sup}}{b_{\inf}-\chi\mu}}{(b_{\inf}-\chi\mu)u^+_{\chi \inf}}+\varepsilon, \quad \forall\ t\geq T_{1,\varepsilon}, \ \forall\ t_0\in\R.
$$
This completes the proof of \eqref{f-2} for the case $n=1$. Next, let us suppose as induction hypothesis that \eqref{f-2} holds up some $n\geq 1$. We show that \eqref{f-2} also holds for $n+1$. Indeed, using the last inequality in \eqref{f-3}, we may suppose that for every $0<\tilde{\varepsilon}\ll 1$,
\begin{equation}\label{f-3'}
\begin{split}
&\|(\lambda (v-v^+_{\chi})+\nabla \ln(u^+_{\chi})\nabla(v-v^+_{\chi}))(\cdot,t+t_0;t_0,u_0)\|_{\infty}\cr
\leq & C_1(\chi)\mu u^+_{\chi \sup}(t+t_0)\|(U-1)(\cdot,t+t_0;t_0,u_0)\|_{\infty}\cr
\leq &\mu C_1(\chi) u^+_{\chi \sup}(t+t_0)\Big(\frac{\chi\mu C_{1}(\chi)u^+_{\sup}}{(b_{\inf}-\chi\mu)u^+_{\chi \inf}}\Big)^{n}\frac{a_{\sup}}{(b_{\inf}-\chi\mu)u^+_{\chi \sup}} +\tilde{\varepsilon}, \quad \forall\ t\geq  T_{n,\tilde \varepsilon}, \ x\in\R^N,\ t_0\in\R,
\end{split}
\end{equation}
for some $T_{n,\tilde \varepsilon}\gg 1 $. Therefore, similar arguments as in the case of $n=1$ from \eqref{f-5} to \eqref{f-11} by replacing  the expression $ \frac{\mu C_1(\chi) a_{\sup}}{b_{\inf}-\chi\mu} $ with $ \mu C_1(\chi) u^+_{\chi \sup}\Big(\frac{\chi\mu C_{1}(\chi)u^+_{\sup}}{(b_{\inf}-\chi\mu)u^+_{\chi \inf}}\Big)^{n}\frac{a_{\sup}}{(b_{\inf}-\chi\mu)u^+_{\chi \sup}} +\tilde{\varepsilon}$,  yield that
\begin{equation*}
\begin{split}
\|U(\cdot,t+t_0;t_0,u_0)-1\|_{\infty}\leq & \frac{\chi C_1(\chi)\mu u^+_{\chi \sup}\Big(\frac{\chi\mu C_{1}(\chi)u^+_{\sup}}{(b_{\inf}-\chi\mu)u^+_{\chi \inf}}\Big)^{n}\frac{a_{\sup}}{(b_{\inf}-\chi\mu)u^+_{\chi \sup}}+ \tilde{\varepsilon}}{(b_{\inf}-\chi\mu)u^+_{\chi \inf}}+ \tilde{\varepsilon} \cr
=& \frac{a_{\sup}\Big(\frac{\chi\mu C_{1}(\chi)u^+_{\sup}}{(b_{\inf}-\chi\mu)u^+_{\chi \inf}}\Big)^{n+1}}{(b_{\inf}-\chi\mu)u^+_{\chi \sup}} + \left(1+\frac{1}{(b_{\inf}-\chi\mu)u^+_{\chi\inf}}\right)\tilde{\varepsilon}, \quad \forall\ t\geq T_{n+1,\tilde{\varepsilon}}, \ \forall\ t_0\in\R,
\end{split}
\end{equation*}
for some $T_{n+1,\tilde{\varepsilon}}\gg 1$.  Which shows that \eqref{f-2} also holds for n+1.

If $u_{\chi}^+(x,t)=u_{\chi}^+(t)$, then using \eqref{f-1'} instead of \eqref{f-1} in the proof of the general case given above,  \eqref{f-5} and \eqref{f-6} become
\begin{equation}\label{f-5'}
U_t\leq \Delta U -\chi\nabla U\nabla v + \Big(\chi\mu\|V-1\|_{\infty}  U + (b-\chi\mu)U(1-U)\Big)u^+_{\chi}(t),
\end{equation}
and
\begin{equation}\label{f-6'}
U_t\geq \Delta U -\chi\nabla U\nabla v + \Big(-\chi\mu\|V-1\|_{\infty}  U + (b-\chi\mu)U(1-U)\Big)u^+_{\chi}(t).
\end{equation}
Observe that $\|V(\cdot,t+t_0;t_0,u_0)-1\|_{\infty}\leq \frac{a_{\sup}}{(b_{\inf}-\chi\mu)u^+_{\chi\inf}}$ for every $t\geq0$, $t_0\in\R$. Hence, by considering  $\underline{U}_1(t)$ and $\overline{U}_1(t)$ solutions of the ODEs
\begin{equation*}
\begin{cases}
\frac{d\underline{U}}{dt}=\Big( -\chi\mu\frac{a_{\sup}}{(b_{\inf}-\chi\mu)u^+_{\chi\inf}}\underline{U} +(b_{\inf}-\chi\mu)\underline{U}(1-\underline{U}) \Big)u^+_{\chi}(t+t_0)\cr
\underline{U}(0)=\min\{\frac{u_{0\inf}}{u^+_{\chi\sup}}, 1\}
\end{cases}
\end{equation*}
and
\begin{equation*}
\begin{cases}
\frac{d\overline{U}}{dt}=\Big( \chi\mu\frac{a_{\sup}}{(b_{\inf}-\chi\mu)u^+_{\chi\inf}}\overline{U} +(b_{\inf}-\chi\mu)\overline{U}(1-\overline{U}) \Big)u^+_{\chi}(t+t_0)\cr
\overline{U}(0)=\max\{\frac{u_{0\sup}}{u^+_{\chi\inf}}, 1+\Big(\frac{\chi\mu}{b_{\inf}-\chi\mu}\Big)\frac{a_{\sup}}{(b_{\inf}-\chi\mu)u^+_{\chi\inf}}\},
\end{cases}
\end{equation*}
similar arguments as those used in the general case, \eqref{f-7}-\eqref{f-11}, yield that \eqref{f-2'} also holds.  This completes the proof of the Theorem.
\end{proof}

\begin{rk}
We note that the type of comparison principles arguments used in the proof Theorem \ref{New-asymp-lem 1} have been used in the literature to study the stability of constant equilibria solution of some chemotaxis models (see \cite{BlLa2016,SaSh1,StTeWi2014,TeWi2012N}). However, the arguments as presented in these works can not be directly applied to \eqref{P} due to the heterogeneity of the underlying source functions. Hence, a very careful and nontrivial  refinement of the technique is required to handle the stability of the positive entire solutions of \eqref{P} in the general heterogeneous media.

\end{rk}

We now present the proof of Theorem \ref{Stability of positive space homogeneous steady state solution}, which is based on the previous result.

Let $\tilde{U}(x,t;t_0,u_0)=U(x,t;t_0;u_0)-1$ and $\tilde{V}(x,t;t_0,u_0)=V(x,t;t_0,u_0)-1$. Then it follows from \eqref{f-1} that  $\tilde{U}(x,t;t_0,u_0)$ satisfies
\begin{align}\label{h-01}
\begin{split}
\tilde U_t=&\Delta \tilde U+\nabla\tilde U\cdot \nabla(2\ln(u^+_{\chi})-\chi v)-(b(x,t)-\chi\mu)u^+_{\chi}(t)\tilde U\cr
&+ \chi U \Big(\lambda (v-v^+_{\chi})+\nabla \ln(u^+_{\chi})\cdot\nabla(v-v^+_{\chi}) \Big)-(b(x,t)-\chi\mu)u^+_{\chi} \tilde U^2.
\end{split}
\end{align}
Let $\Phi_{\chi}(t,s)$ be the solution operator  in $C^b_{\rm unif}(\R^N)$ of
\begin{equation}\label{h-02}
u_t=\Delta u + \nabla u\cdot \nabla(2\ln(u^+_{\chi})-\chi v)-(b(x,t)-\chi\mu)u^+_{\chi}u.
\end{equation}
Thus, by the comparison principle for parabolic equations, we have
\begin{equation}\label{h-03}
\|\Phi_{\chi}(t,s)\|\leq e^{-(t-s)(b_{\inf}-\chi\mu)u^+_{\chi \inf}}, \quad \forall\ t-s\ge 0.
\end{equation}

\begin{proof}[Proof of Theorem \ref{Stability of positive space homogeneous steady state solution}] We shall give the proof of the general case. The proof of the particular case follows similar arguments. We suppose that $0<\chi<\chi_0$, where $\chi_0$ is given by \eqref{chi_0 def}.  Hence, by definition of $ \chi_0$, there is  a positive entire solution of \eqref{P} $ (u^+_{\chi}(x,t),v^+_{\chi}(x,t))$ satisfying
$$
(\tilde{H}) : \quad \frac{\chi\mu C_{1}(\chi)u^+_{\chi\sup}}{(b_{\inf}-\chi\mu)u^+_{\chi \inf}}<1.
$$

\noindent{\bf Exponential Stability of $(u^+_{\chi}(x,t),v^+_{\chi}(x,t))$:}
By Theorem \ref{New-asymp-lem 1} we may suppose that  there are $t_{n}\gg1 $ with $0<t_{n}<t_{n+1}, $ such that
\begin{equation}\label{h-10}
\|\tilde{U}(\cdot,t+t_0;t_0,u_0)\|_{\infty}=\|U(\cdot,t+t_0;t_0,u_0)-1\|_{\infty}\leq 2\Big( \frac{\chi\mu C_{1}(\chi)u^+_{\sup}}{(b_{\inf}-\chi\mu)u^+_{\chi \inf}}\Big)^{n} ,\quad \forall\ t\geq t_n, \ t_0\in\R.
\end{equation}
By the variation of constant formula, it follows from \eqref{h-01} that for every $ t\geq 0$,
\begin{equation}\label{h-12}
\tilde U(\cdot,t+t_n+t_0;t_0,u_0)= I_{1,n}(t;t_0)+\chi I_{2,n}(t;t_0) - I_{3,n}(t,t_0),
\end{equation}
where
\begin{equation*}\label{I1n eq}
I_{1,n}(t,t_0):=\Phi_{\chi}(t+t_n+t_0;t_n+t_0)\tilde{U}(\cdot,t_n+t_0;t_0,u_0), \forall t\geq 0, \forall\ n\geq 1,
\end{equation*}
\begin{equation*}\label{I2n eq}
I_{2,n}(t,t_0):=\int_{0}^t\Phi_{\chi}(t+t_n+t_0,s+t_n+t_0) \Big( U \big(\lambda (v-v^+_{\chi})+\nabla \ln(u^+_{\chi})\cdot\nabla(v-v^+_{\chi})\big)\Big)(\cdot,s+t_n+t_0) ds,
\end{equation*}
and
\begin{equation*}\label{I3n eq}
I_{3,n}(t,t_0):=\int_0^{t}\Phi_{\chi}(t+t_n+t_0,s+t_n+t_0)(b-\chi\mu)u^+_{\chi}\tilde{U}^2(\cdot,s+t_n+t_0)ds.
\end{equation*}
Next, it follows from \eqref{h-10} and \eqref{h-03}  that for every $n\geq 1$, $t_0\in\R$, and $t\geq 0$,
 \begin{equation}\label{h-13}
 \begin{split}
\|\Phi_{\chi}(t+t_n+t_0;t_n+t_0)\tilde{U}(\cdot,t_n+t_0;t_0,u_0)\|_{\infty}\leq & e^{-t(b_{\inf}-\chi\mu)u^+_{\chi \inf}}\|\tilde{U}(\cdot,t_n+t_0;t_0,u_0)\|_{\infty}\cr
\leq & \underbrace{2\Big( \frac{\chi\mu C_{1}(\chi)u^+_{\sup}}{(b_{\inf}-\chi\mu)u^+_{\chi \inf}}\Big)^n}_{:=K_{1,n}}e^{-t(b_{\inf}-\chi\mu)u^+_{\chi \inf}}.
\end{split}
 \end{equation}
 Next, for every $0\leq s\leq t$, $n\geq 1$, and $t_0\in\R$, we have
 \begin{equation}\label{h-14}
 \begin{split}
& \|\Phi_{\chi}(t+t_n+t_0,s+t_n+t_0) (( U \Big(\lambda (v-v^+_{\chi})+\nabla \ln(u^+_{\chi})\cdot \nabla(v-v^+_{\chi}))(s+t_n+t_0) \Big)\|_{\infty}\cr
\leq &  e^{-(t-s)(b_{\inf}-\chi\mu)u^+_{\chi \inf}}\|\Big(U \Big(\lambda (v-v^+_{\chi})+\nabla \ln(u^+_{\chi})\cdot \nabla(v-v^+_{\chi})\Big)\Big)(s+t_n+t_0)  \|_{\infty}\cr
\leq & \Big(1+2\Big( \frac{\chi\mu C_{1}(\chi)u^+_{\sup}}{(b_{\inf}-\chi\mu)u^+_{\chi \inf}}\Big)^n\Big)\frac{\| \Big(\lambda (v-v^+_{\chi})+\nabla \ln(u^+_{\chi})\cdot \nabla(v-v^+_{\chi})\Big)(s+t_n+t_0)\|_{\infty}}{e^{(t-s)(b_{\inf}-\chi\mu)u^+_{\chi \inf}}}\cr
 \leq & \Big(1+2\Big( \frac{\chi\mu C_{1}(\chi)u^+_{\sup}}{(b_{\inf}-\chi\mu)u^+_{\chi \inf}}\Big)^n\Big)\Big(1 + \frac{ C_0(\chi)\sqrt{N}}{u^+_{\chi \inf}\sqrt{\lambda}}\Big)\mu\|(u-u^+_{\chi})(s+t_n+t_0)\|_{\infty}e^{-(t-s)(b_{\inf}-\chi\mu)u^+_{\chi \inf}}\cr
 \leq & \underbrace{\Big(1+2\Big( \frac{\chi\mu C_{1}(\chi)u^+_{\sup}}{(b_{\inf}-\chi\mu)u^+_{\chi \inf}}\Big)^n\Big)\Big(1 + \frac{ C_0(\chi)\sqrt{N}}{u^+_{\chi \inf}\sqrt{\lambda}}\Big)\mu u^+_{\chi \sup}}_{:=K_{2,n}}\|\tilde{U}(\cdot,s+t_n+t_0)\|e^{-(b_{\inf}-\chi\mu)u^+_{\chi \inf}(t-s)}
 \end{split}
 \end{equation}
 We also have
 \begin{equation}\label{h-15}
 \begin{split}
& \|\Phi_{\chi}(t+t_n+t_0,s+t_n+t_0)(b-\chi\mu)u^+_{\chi}\tilde{U}^2(\cdot,s+t_n+t_0)\|_{\infty}\cr
\leq &  2 \underbrace{(b_{\sup}-\chi\mu)u^+_{\chi \sup}\Big( \frac{\chi\mu C_{1}(\chi)u^+_{\sup}}{(b_{\inf}-\chi\mu)u^+_{\chi \inf}}\Big)^n}_{:=K_{3,n}}\|\tilde{U}(\cdot,s+t_n+t_0)\|_{\infty}e^{-(t-s)(b_{\inf}-\chi\mu)u^+_{\inf}}.
 \end{split}
 \end{equation}
 Thus, it follows from \eqref{h-12}, \eqref{h-13}, \eqref{h-14}, and \eqref{h-15} that
 \begin{equation*}
 \begin{split}
& \|\tilde U(\cdot,t+t_n+t_0;t_0,u_0)\|_{\infty}\cr
\leq& K_{1,n}e^{-(b_{\inf}-\chi\mu)t}+(\chi K_{2,n}+K_{3,n})\int_0^{t}e^{-(t-s)(b_{\inf}-\chi\mu)u^+_{\chi\sup}}\|\tilde U(\cdot,s+t_n+t_0;t_0,u_0)\|_{\infty}ds,
 \end{split}
 \end{equation*}
 which is equivalent to
 \begin{equation*}\begin{split}
& e^{t(b_{\inf}-\chi\mu)u^+_{\chi\sup}}\|\tilde U(\cdot,t+t_n+t_0;t_0,u_0)\|_{\infty}\cr
\leq& K_{1,n}+(\chi K_{2,n}+K_{3,n})\int_0^{t}e^{s(b_{\inf}-\chi\mu)u^+_{\chi\sup}}\|\tilde U(\cdot,s+t_n+t_0;t_0,u_0)\|_{\infty}ds, \forall\ t\geq 0.
 \end{split}
 \end{equation*}
 Therefore, by Grownwall's inequality, we obtain that
 \begin{equation*}
 e^{t(b_{\inf}-\chi\mu)u^+_{\chi\sup}}\|\tilde U(\cdot,t+t_n+t_0;t_0,u_0)\|_{\infty}\leq K_{1,n}e^{(\chi K_{2,n}+K_{3,n})t}, \quad \forall\ t\geq 0.
 \end{equation*}
 That is \begin{equation}\label{h-16}
|\tilde U(\cdot,t+t_n+t_0;t_0,u_0)\|_{\infty}\leq K_{1,n}e^{-\big((b_{\inf}-\chi\mu)u^+_{\chi\sup}-\chi K_{2,n}-K_{3,n}\big)t}, \quad \forall\ t\geq 0.
 \end{equation}
 By $(\tilde{H})$, we have
 $$
\lim_{n\to\infty}K_{1,n}=\lim_{n\to\infty}K_{3,n}=0 \quad \text{and}\quad \lim_{n\to\infty}K_{2,n}=\Big(1+\frac{ C_0(\chi)\sqrt{N}}{\sqrt{\lambda}u^+_{\chi\inf}}\Big)\mu u^+_{\chi \sup}=\mu C_{1}(\chi)u^+_{\chi\sup}.
 $$
Since $(\tilde{H})$ holds, then there is $n_0 \gg 1$ such that
$$
\alpha_{\chi}:=\sup_{n\geq n_0}((b_{\inf}-\chi\mu)u^+_{\chi\sup}-\chi K_{2,n}-K_{3,n})>0
$$
This combined with \eqref{h-16} yield that
$$\|u(\cdot,t+t_{n_0}+t_0;t_0,u_0)-u^+_{\chi}(t+t_{n_0}+t_0)\|_{\infty}\leq u^+_{\chi\sup}K_{1,n_{0}}e^{-t\alpha_{\chi}}\quad \forall \, t\ge 0, $$
which implies that  $(u^+_{\chi}(x,t),v^+_{\chi}(x,t))$ is exponentially stable.

\medskip

\noindent {\bf Uniqueness of $(u^+_{\chi}(x,t),v^+_{\chi}(x,t))$: } Let $(\tilde{u}^+_{\chi}(x,t),\tilde{v}^+_{\chi}(x,t))$ be a positive entire solution of \eqref{P}. Then, since $0<\tilde{u}^+_{\chi\inf}\leq \tilde{u}^+_{\chi\sup}<\infty $,  it follows from the exponential stability of $(u^+_{\chi}(x,t),v^+_{\chi}(x,t))$ that there is a positive constant $K$ depending only on $\tilde{u}^+_{\chi\inf}, \tilde{u}^+_{\chi\sup}, u^+_{\chi\inf}$, and $ u^+_{\chi\sup} $, such that
\begin{equation*}
\begin{split}
&\|\tilde{u}^+_{\chi}(\cdot,t)-u^+_{\chi}(\cdot,t)\|_{\infty}\cr
 = & \|\tilde{u}^+_{\chi}(\cdot,n+(t-n);t-n,u^+_{\chi}(\cdot,t-n))-u^+_{\chi}(\cdot,n+(t-n);t-n,u^+_{\chi}(\cdot,t-n))\|_{\infty}\cr
 \leq & K e^{-n \alpha_{\chi}}, \quad \forall\ n\geq 1.
\end{split}
\end{equation*}
 Letting $n\to\infty$ in the last inequality yields that $u^+_{\chi}(x,t)\equiv \tilde{u}^+_{\chi}(x,t)$. This completes the proof of Theorem \ref{Stability of positive space homogeneous steady state solution}.
\end{proof}

\medskip

Next, we present the proof of Theorem \ref{perturbation effect}.

\medskip

\begin{proof}[Proof of Theorem \ref{perturbation effect}]
The proof of this result follows from the arguments used in the proof of Theorem \ref{New-asymp-lem 1}. So, lengthy detail will be avoided.

(i) Let $U(x,t)=\frac{u^+_{\chi}(x,t)}{u^+_0(x,t)}$ for $x\in\R^N$ and $t\in\R.$
Then, it follows from similar arguments leading to \eqref{f-1}, that $U(x,t)$ satisfies
\begin{equation}\label{h-17}
U_{t}=\Delta U +\nabla U\cdot \nabla (2\ln(u_0^+)-\chi v^+_{\chi})-\chi U\Big(\Delta v^+_{\chi} +\nabla \ln(u_0)\cdot \nabla v^+_{\chi}\Big) +b(x,t)u^+_{0}(x,t)U(1-U).
\end{equation}
Observe that
\begin{equation*}
\begin{split}
\|\Big(\Delta v^+_{\chi} +\nabla \ln(u_0)\cdot \nabla v^+_{\chi}\Big)(\cdot,t)\|_\infty\leq & \|\Delta v(\cdot,t)\|_{\infty}+\|\Big(\nabla \ln(u_0)\cdot \nabla v^+_{\chi}\Big)(\cdot,t)\|_{\infty}\cr
\leq & 2\mu \|u^+_{\chi}(\cdot,t)\|_{\infty}+\|\nabla\ln(u_0^+(\cdot,t))\|_{\infty}\|\nabla v^+_{\chi}(\cdot,t)\|_{\infty}\cr
\leq &\mu\Big( 2 +\frac{\sqrt{N}}{\sqrt{\lambda}}\sup_{t\in\R}\|\nabla\ln(u^+_0(\cdot,t))\|_{\infty}\Big)\|u^+_{\chi}(\cdot,t)\|_{\infty}\cr
\leq & \underbrace{\Big( 2 +\frac{\sqrt{N}}{\sqrt{\lambda}}\sup_{t\in\R}\|\nabla\ln(u^+_0(\cdot,t))\|_{\infty}\Big)}_{:=K}\frac{\mu a_{\sup}}{b_{\inf}-\chi\mu}.
\end{split}
\end{equation*}
Thus, the last inequality combined with \eqref{h-17} yields that
\begin{equation}\label{h-18}
U_t\leq \Delta U +\nabla U\cdot \nabla (2\ln(u_0^+)-\chi v^+_{\chi}) +\frac{\chi\mu a_{\sup}K}{b_{\inf}-\chi\mu}U+b(x,t)u^+_{0}(x,t)U(1-U), \quad x\in\R^N, \ t\in\R,
\end{equation}
and
\begin{equation}\label{h-19}
U_t\geq \Delta U +\nabla U\cdot \nabla (2\ln(u_0^+)-\chi v^+_{\chi}) -\frac{\chi\mu a_{\sup}K}{b_{\inf}-\chi\mu}U+b(x,t)u^+_{0}(x,t)U(1-U)\quad \ x\in\R^N, \ t\in\R.
\end{equation}
As in the proof of \eqref{f-9}, letting  $\overline{U}(t)$ denote the solution of the ODE
$$
\begin{cases}
\frac{d\overline{U}}{d t}= \frac{\chi\mu a_{\sup}K}{b_{\inf}-\chi\mu}\overline{U}+b_{\inf}u^+_{0\inf}\overline{U}(1-\overline{U}), t>0\cr
\overline{U}(0)=\max\{\sup_{x,t}U(x,t), \frac{b_{\inf}u^+_{0\inf}+\frac{\chi\mu a_{\sup}K}{b_{\inf}-\chi\mu}}{b_{\inf}u^+_{0\inf}}\},
\end{cases}
$$
it follows from \eqref{h-18} and comparison principle for parabolic equations that
\begin{equation}\label{h-20}
U(x,t+t_0)\leq \overline{U}(t),\quad \forall \ t\geq 0, \forall\ t_0\in\R, \ \forall\ x\in\R^N.
\end{equation}
Observe  that
$$
\lim_{t\to\infty}\overline{U}(t)= \frac{b_{\inf}u^+_{0\inf}+\frac{\chi\mu a_{\sup}K}{b_{\inf}-\chi\mu}}{b_{\inf}u^+_{0\inf}}.
$$
Thus, it follows from \eqref{h-20} that
\begin{equation}\label{h-21}
U(x,t)\leq \lim_{t_0\to-\infty}\overline{U}(t-t_0)=\frac{b_{\inf}u^+_{0\inf}+\frac{\chi\mu a_{\sup}K}{b_{\inf}-\chi\mu}}{b_{\inf}u^+_{0\inf}}, \forall \ x\in \R^N, \ \forall\ t\in\R.
\end{equation}
On the other hand, letting $\underline{U}(t)$ be the solution  of the ODE
$$
\begin{cases}
\frac{d\underline{U}}{d t}= -\frac{\chi\mu a_{\sup}K}{b_{\inf}-\chi\mu}\underline{U}+b_{\inf}u^+_{0\inf}\underline{U}(1-\underline{U}), t>0\cr
\underline{U}(0)=\min\{\inf_{x,t}U(x,t), 1\},
\end{cases}
$$
it follows from \eqref{h-19} and comparison principle for parabolic equations that
\begin{equation}\label{h-20'}
U(x,t+t_0)\geq \underline{U}(t),\quad \forall \ t\geq 0, \forall\ t_0\in\R, \ \forall\ x\in\R^N.
\end{equation}
Observe also that
$$
\lim_{t\to\infty}\underline{U}(t)= \frac{(b_{\inf}u^+_{0\inf}-\frac{\chi\mu a_{\sup}K}{b_{\inf}-\chi\mu})_{+}}{b_{\inf}u^+_{0\inf}}.
$$
Hence inequality \eqref{h-20'} yields that
\begin{equation}\label{h-22}
U(x,t)\geq\lim_{t_0\to-\infty}\underline{U}(t-t_0)= \frac{(b_{\inf}u^+_{0\inf}-\frac{\chi\mu a_{\sup}K}{b_{\inf}-\chi\mu})_{+}}{b_{\inf}u^+_{0\inf}}\geq 1- \frac{\chi\mu a_{\sup}K}{(b_{\inf}-\chi\mu)b_{\inf}u^+_{0\inf}}, \quad \forall\ x\in\R^N, \ t\in\R.
\end{equation}
Therefore, it follows from \eqref{h-21} and \eqref{h-22} that
$$
\|U(x,t)-1\|_{\infty}\leq \frac{\chi\mu a_{\sup}K}{(b_{\inf}-\chi\mu)b_{\inf}u^+_{0\inf}}, \quad \forall\ t\in\R.
$$
This implies that
$$
\|u^+_{\chi}(x,t)-u^+_0(x,t)\|_{\infty}\leq \frac{\chi\mu a_{\sup}u^+_{0\sup}K}{(b_{\inf}-\chi\mu)b_{\inf}u^+_{0\inf}}, \quad \forall\ t\in\R, \ x\in\R^N.
$$
Thus inequality \eqref{small chemotatic pert eq1}  follows.

(ii) Next, we prove inequality \eqref{small chemotatic pert eq2}. For, let $u_0\in C^b_{\rm unif}(\R^N)$ with $u_{0\inf}>0$. Observe that  $\|u_{\chi}(\cdot,t+t_0;t_0,u_0)\|_{\infty}\leq \max\{u_{0\sup},\frac{a_{\sup}}{b_{\inf}-\chi\mu}\}$ for every $t\geq 0$ and $t_0\in\R$ and $\|\nabla v_{\chi}(\cdot,t+t_0;t_0,u_0)|_{\infty}\leq \frac{\mu\sqrt{N}}{\sqrt{\lambda}}\|u_{\chi}(\cdot,t+t_0;t_0,u_0)\|_{\infty}$.
Hence, for every $t\geq 0$, $\ t_0\in\R,$ and $\ 0\leq \chi<b_{\inf}$, it follows from \eqref{eq-f-0-02} that
\begin{equation*}
\begin{split}
&\|u_{\chi}(\cdot,t+t_0;t_0,u_0)-u_{0}(\cdot,t+t_0;t_0,u_0)\|_{\infty}\cr
\leq & \chi \int_0^t\|T(t-s)\nabla\cdot((u_{\chi}\nabla v_{\chi})(\cdot,s+t_0;t_0,u_0))\|_{\infty}ds \cr
& + \int_0^t\|T(t-s)\Big((a+1+b(u_{\chi}-u_0))(u_{\chi}-u_0)) \Big)(\cdot,s+t_0;t_0,u_0)\|_{\infty}ds\cr
\leq & \frac{N\chi}{\sqrt{\pi}}\int_0^t\frac{e^{-(t-s)}}{\sqrt{t-s}}\|((u_{\chi}\nabla v_{\chi})(\cdot,s+t_0;t_0,u_0))\|_{\infty}ds\cr
&+\int_0^te^{-(t-s)}\|\Big((a+1+b(u_{\chi}-u_0))(u_{\chi}-u_0)) \Big)(\cdot,s+t_0;t_0,u_0)\|_{\infty}\cr
\leq &\frac{N\chi\sqrt{N}}{\sqrt{\pi\lambda}}\max\{u_{0\sup}^2,\Big(\frac{a_{\sup}}{b_{\inf}-\chi\mu}\Big)^2\}\int_0^t\frac{e^{-(t-s)}}{\sqrt{t-s}}ds\cr
&+\Big(a_{\sup}+1+b_{\sup}\max\{u_{0\sup},\frac{a_{\sup}}{b_{\inf}-\chi\mu}\}\Big)\int_0^t\|(u_{\chi}-u_0)(\cdot,s+t_0;t_0,u_0)\|_{\infty}\cr
\leq & \chi C_{1}\sqrt{t} +C_{2}\int_0^t\|(u_{\chi}-u_0)(\cdot,s+t_0;t_0,u_0)\|_{\infty},
\end{split}
\end{equation*}
where $C_1:=2\frac{N\sqrt{N}}{\sqrt{\pi\lambda}}\max\{u_{0\sup}^2,\Big(\frac{a_{\sup}}{b_{\inf}-\chi\mu}\Big)^2\}$ and $C_2:=\Big(a_{\sup}+1+b_{\sup}\max\{u_{0\sup},\frac{a_{\sup}}{b_{\inf}-\chi\mu}\}\Big)$.
Thus, by Grownwall inequality, we have that
\begin{equation}\label{h-23}
\|u_{\chi}(\cdot,t+t_0;t_0,u_0)-u_{0}(\cdot,t+t_0;t_0,u_0)\|_{\infty}\leq \chi C_{1}\sqrt{t}e^{C_2 t}, \quad \forall t\geq 0, \forall \ t_0\in\R, \ \forall\ 0<\chi<b_{inf}.
\end{equation}
In particular, taking $t=1$, we obtain for every $t_0\in\R$ and $0\leq \chi<b_{\inf}$,
\begin{equation}\label{h-24}
\|\frac{u_{\chi}(\cdot,1+t_0;t_0,u_0)}{u_0(\cdot,1+t_0;t_0,u_0)}-1 \|_{\infty}\leq\frac{1}{C_3}\|u_{\chi}(\cdot,1+t_0;t_0,u_0)- u_0(\cdot,1+t_0;t_0,u_0) \|_{\infty}\leq \chi\frac{C_1 e^{C_2}}{C_3},
\end{equation}
where $C_3:=\inf_{t\geq t_0, x\in\R^N, t_0\in\R}u_{0}(x,t;t_0,u_0)>0$.

Taking $\chi=0$ in \eqref{eq-f-0-02} and differentiating  both sides with respect to space variable yield that
\begin{equation*}
\begin{split}
C_4:=&\sup_{t\geq0, t_0\in\R}\Big(2 +\frac{N}{\sqrt{\lambda}}\|\nabla \ln(u_0(\cdot,t+1+t_0;t_0,u_0)) \|_{\infty} \Big)\cr
\leq&\Big(2+\frac{N\sqrt{N}}{\sqrt{\lambda}} (e^{-1}+1+a_{\sup}+b_{\sup}\max\{u_{0\sup},\frac{a_{\sup}}{b_{\inf}}\})\max\{u_{0\sup},\frac{a_{\sup}}{b_{\inf}}\}\Big) <\infty.
\end{split}
\end{equation*}
By setting $U(x,t)=\frac{u_{\chi}(x,t+1+t_0;t_0,u_0)}{u_0(x,t+1+t_0;t_0,u_0)}$ for $x\in\R^N$ and $t\geq 0$, similar arguments leading to \eqref{h-20} and \eqref{h-20'} yield that
\begin{equation}\label{h-25}
\underline{U}(t)\leq \frac{u_{\chi}(x,t+1+t_0;t_0,u_0)}{u_0(x,t+1+t_0;t_0,u_0)}\leq \overline{U}(t),\quad \forall \ x\in\R^N, \ \forall\ t\geq 0, \ t_0\in\R,
\end{equation}
where $\overline{U}(t)$ and $\underline{U}(t)$ are solutions of the ODE's
$$
\begin{cases}
\frac{d\underline{U}}{d t}= -\frac{\chi\mu a_{\sup}C_4}{b_{\inf}-\chi\mu}\underline{U}+b_{\inf}C_3\underline{U}(1-\underline{U}), t>0\cr
\underline{U}(0)=\min\Big\{\Big(1-\chi\frac{C_1e^{C_2}}{C_3}\Big)_{+}, \Big(\frac{b_{\inf}C_{3}-\frac{\chi\mu a_{\sup}C_4}{b_{\inf}-\chi\mu}}{b_{\inf}C_{3}}\Big)_+\Big\}
\end{cases}
$$
and
$$
\begin{cases}
\frac{d\overline{U}}{d t}= \frac{\chi\mu a_{\sup}C_{4}}{b_{\inf}-\chi\mu}\overline{U}+b_{\inf}C_3\overline{U}(1-\overline{U}), t>0\cr
\overline{U}(0)=\max\Big\{1+\chi\frac{C_1e^{C_2}}{C_3}, \frac{b_{\inf}C_{3}+\frac{\chi\mu a_{\sup}C_4}{b_{\inf}-\chi\mu}}{b_{\inf}C_{3}}\Big\},
\end{cases}
$$
respectively. Note that we have used \eqref{h-24} and comparison principle for parabolic equations to obtain \eqref{h-25}. It is easy to see that $\underline{U}(0)\leq \underline{U}(t)$ and $\overline{U}(t)\leq \overline{U}(0)$ for every $t\geq 0$. Therefore, it follows from inequality \eqref{h-25} that
$$
\underline{U}(0)\leq \frac{u_{\chi}(x,t+1+t_0;t_0,u_0)}{u_0(x,t+1+t_0;t_0,u_0)} \leq \overline{U}(0),\quad \forall \ x\in\R^N, \ \forall\ t\geq 0, \ t_0\in\R,
$$
which implies that
\begin{equation*}
\begin{split}
&\|u_{\chi}(\cdot,t+1+t_0;t_0,u_0)-u_0(\cdot,t+1+t_0;t_0,u_0)\|_{\infty}\cr
\leq & \chi\max\Big\{u_{0\sup},\frac{a_{\sup}}{b_{\inf}}\Big\}\max\Big\{\frac{C_1e^{C_2}}{C_3}, \frac{\mu a_{\sup}C_4}{(b_{\inf}-\chi\mu)b_{\inf}C_{3}} \Big\}, \ \quad \ \ \forall\ t\geq 0, \ t_0\in\R.
\end{split}
\end{equation*}
This combined with \eqref{h-23} yields \eqref{small chemotatic pert eq2}.
\end{proof}

\end{document}